\documentclass[12pt,a4paper]{article}
\usepackage{enumerate}
\usepackage{theorem}
\usepackage{graphicx}
\newtheorem{lemma}{Lemma}
\newtheorem{proposition}[lemma]{Proposition}
\newtheorem{theorem}[lemma]{Theorem}
\newtheorem{corollary}[lemma]{Corollary}
{\theorembodyfont{\upshape}\newtheorem{probl}[lemma]{Problem}
{\theorembodyfont{\upshape}}
{\theorembodyfont{\upshape}\newtheorem{remarke}[lemma]{Remark}}
{\theorembodyfont{\upshape}\newtheorem{notee}[lemma]{Note}}
{\theorembodyfont{\upshape}\newtheorem{exampl}[lemma]{Example}}
{\theorembodyfont{\upshape}}

\newenvironment{example}{\begin{exampl}}{\hfill$\Box$\end{exampl}}

\newenvironment{remark}{\begin{remarke}}{\hfill$\Box$\end{remarke}}
\newcommand{\C}{{\bf C}}

\newcommand{\N}{{\bf N}}

\newcommand{\R}{{\bf R}}

\newcommand{\rme}{{\rm e}}
\newcommand{\rmd}{{\rm d}}

\newcommand{\cC}{{\cal C}}
\newcommand{\cD}{{\cal D}}

\newcommand{\cG}{{\cal G}}
\newcommand{\cH}{{\cal H}}

\newcommand{\cK}{{\cal K}}
\newcommand{\cL}{{\cal L}}

\newcommand{\cR}{{\cal R}}

\newcommand{\tA}{\tilde{A}}
\newcommand{\tB}{\tilde{B}}
\newcommand{\sig}{\sigma}
\newcommand{\alp}{\alpha}
\newcommand{\bet}{\beta}
\newcommand{\gam}{\gamma}

\newcommand{\lam}{\lambda}
\newcommand{\del}{\delta}
\newcommand{\eps}{\varepsilon}

\newcommand{\tr}{{\rm tr}}
\newcommand{\Dom}{{\rm Dom}}

\newcommand{\Spec}{{\rm Spec}}
\newcommand{\Ess}{{\rm Ess}}

\newcommand{\Ker}{{\rm Ker}}
\newcommand{\Ran}{{\rm Ran}}
\newcommand{\rank}{{\rm rank}}

\newcommand{\norm}{\Vert}
\renewcommand{\Re}{{\rm Re}}
\renewcommand{\Im}{{\rm Im}}

\newcommand{\lin}{{\rm lin}}

\newcommand{\dist}{{\rm dist}}
\newcommand{\Schrodinger}{Schr\"odinger }
\newcommand{\implies}{\Rightarrow}
\newcommand{\la}{{\langle}}
\newcommand{\ra}{{\rangle}}
\newcommand{\pr}{\prime}

\newenvironment{eq}{\begin{equation}}{\end{equation}}

\newcommand{\eqref}[1]{(\ref{#1})}
\newcommand{\nat}{\natural }

\newenvironment{proof}{\textbf{Proof}}{\hfill$\Box$}

\newcommand{\mtrx}[1]{\left(\begin{array}{cc}#1 \end{array}\right)}

\newcommand{\str}{\rule{0em}{1.5ex}}
\newcommand{\move}[1]{}
\usepackage{amssymb}
\usepackage{hyperref}
\newcommand{\view}{}
\newcommand{\draft}[1]{#1}
\title{Sectorial perturbations\\
of self-adjoint matrices and operators}
\author{E B Davies}
\date{8 June 2012}
\begin{document}
\view\maketitle

\begin{abstract}
\view This paper considers $N\times N$ matrices of the form $A_\gam =A+ \gam B$, where $A$ is self-adjoint, $\gam\in\C$ and $B$ is a non-self-adjoint perturbation of $A$. We obtain some monodromy-type results relating the spectral behaviour of such matrices in the two asymptotic regimes $|\gam|\to\infty$ and $|\gam|\to 0$ under certain assumptions on $B$. We also explain some properties of the spectrum of $A_\gam$ for intermediate sized $\gam$ by considering the limit $N\to\infty$, concentrating on properties that have no self-adjoint analogue. A substantial number of the results extend to operators on infinite-dimensional Hilbert spaces.
\end{abstract}

AMS subject classifications:


\section{Introduction}\label{intro}

Let $A$ be a (possibly unbounded) self-adjoint operator acting in the Hilbert space $\cH$ and let $A_\gam=A+ \gam B$ where $\gam\in\C$ and $B$ is a bounded operator on $\cH$. Many papers have been written about the spectral properties of the self-adjoint operators $A_\gam$ when $B=B^\ast$ and $\gam\in\R$, the main techniques used including variational inequalities and perturbation expansions; see \cite{Liaw,LT,Simon} and many further references there. In this paper we concentrate on more general $B$ and assume that $\gam$ is complex. Our main concern is to describe phenomena that have no self-adjoint analogues, an issue that has been curiously neglected. A recent paper of Rana and Wojtylak, \cite{RW}, is closer to this one, but there is little technical overlap. The interplay between the asymptotic regimes $|\gam|\to 0$ and $|\gam|\to\infty$ is a main focus of interest, but we also explore some spectral phenomena that arise for intermediate values of $\gam$.

As well as being of intrinsic interest, operators of this type are relevant to non-self-adjoint \Schrodinger and wave equations, for which the evolution is contractive as a function of time. In such situations every eigenvalue of $A_\gam$ lies in an appropriate half-plane and the eigenvalue determines the energy and rate of decay of the associated eigenstate of the system. From Section~\ref{rank1case} onwards we study rank one perturbations. As well as providing a range of phenomena that must be included in a more general theory, this is of direct relevance to the study of non-self-adjoint boundary conditions for \Schrodinger operators in one dimension. The relevant perturbations of the \Schrodinger operators are singular, but, if one considers instead the resolvent operators, the perturbations are rank one and bounded.

General considerations from perturbation theory imply that the set
$\cR$ of $(\gam,\lam)\in \C^2$ such that $\lam$ is an isolated
eigenvalue of $A_\gam$ with finite algebraic multiplicity is a
Riemann surface that may have branch points where the multiplicity
of the eigenvalue is greater than $1$; see \cite{Kato}. If $B$ is relatively compact with respect to $A$ then
\[
\Spec(A_\gam)=\Ess(A)\cup \{\lam:(\gam,\lam)\in \cR\}
\]
for every $\gam\in\C$. Our goal in this paper is to understand how the geometrical structure of $\cR$ depends upon some simple generic assumptions about $A$ and $B$.

In much of the paper we assume that $\cH$ has finite dimension $N$. We assume that $A$ is self-adjoint and that $B$ is sectorial. The coupling constant $\gam$ is restricted by the requirement that $\Im(\la A_\gam f,f\ra)\geq 0$ for all $f\in\cH$; this is equivalent to assuming that $iA_\gam$ is dissipative in a standard sense; see \cite[Section~8.3]{LOTS}. Further assumptions on $B$ are made as necessary. In the particular case $B=B^\ast\geq 0$, which motivated our initial interest, we assume that $0<\arg(\gam)<\pi$. Theorem~\ref{trunc2} and Example~\ref{trunc3} show how a substantial part of the spectrum of a large matrix may sometimes be approximated by using a carefully chosen matrix that is much smaller. Section~\ref{dimdep} focuses on spectral properties of $A_\gam$ that are best understood by considering the limit $N\to\infty$.

\section{Sectorial operators}

The truncation of an operator $A$ on $\cH$ to a closed subspace $\cK$ is defined by $A^\nat = PAP\vert_\cK$, where $P$ is the orthogonal projection of $\cH$ onto $\cK$. We will need the following lemma.

\begin{lemma}\label{trunkspec}
If $A=A^\ast$, $aI\leq A\leq bI$ and $A^\nat$ denotes the truncation of $A$ to $\cK$ then $aI^\nat\leq A^\nat\leq bI^\nat$.
\end{lemma}

\begin{proof}
We use variational methods. The hypotheses imply that
\begin{eqnarray*}
a&\leq& \inf\{\la Af,f\ra:f\in\cH\mbox{ and } \norm f \norm =1\}\\
&\leq & \inf\{\la Af,f\ra:f\in\cK\mbox{ and } \norm f \norm =1\}\\
&=& \inf\{\la A^\nat f,f\ra:f\in\cK\mbox{ and } \norm f \norm =1\}.
\end{eqnarray*}
Therefore $aI^\nat \leq A^\nat$. The other half of the proof is similar.
\end{proof}

A bounded operator $D$ on the Hilbert space $\cH$ is said to be sectorial if there exist `sectorial constants' $\sig_1,\, \sig_2$ such that $-\pi/2< -\sig_1\leq 0 \leq \sig_2<\pi/2$ and
\begin{equation}
\{ \la Df,f\ra:f\in\cH \}\subseteq \{0\}\cup %
\{ z:z\not=0\mbox{ and } -\sig_1\leq \arg(z)\leq \sig_2\}.\label{sectdef}
\end{equation}
The theory of sectorial operators has a long history; see Sections~VI.1.5 and VI.3.1 of \cite{Kato}. The following lemma is adapted from \cite[Theorem VI.3.2]{Kato}, but we include a proof for completeness.

\begin{lemma}\label{sectoriallemma}
If $D$ is a bounded sectorial operator on $\cH$ and $f\in \cH$ then the following are equivalent.
\begin{enumerate}[(i)]
\item $\la Df,f\ra=0$;
\item $\la (D+D^\ast)f,f\ra=0$;
\item $Df=0$;
\item $D^\ast f=0$.
\end{enumerate}
If $\cK=\Ker (D)$ then $\cK$ and $\cK^\perp$ are invariant under $D$ and $D^\ast$. Moreover $D\vert_\cK=D^\ast\vert_\cK=0$. Both  $D\vert_{\cK^\perp}$ and $D^\ast\vert_{\cK^\perp}$ are one-one with ranges that are dense in $\cK^\perp$.  The truncation $D^\nat$ of $D$ to $\cK^\perp$ may be written in the form
\begin{equation}
D^\nat=X^{1/2}(I^\nat+iE) X^{1/2}\label{sectrep1}
\end{equation}
where $X$ is the truncation of $(D+D^\ast)/2$ to $\cK^\perp$, $I^\nat$ is the identity operator on $\cK^\perp$ and  $E$ is a self-adjoint operator on $\cK^\perp$ satisfying
\begin{equation}
-\tan(\sig_1) I^\nat\ \leq E \leq \tan(\sig_2) I^\nat, \label{sectrep2}
\end{equation}
where $\sig_1,\, \sig_2$ are the sectorial constants of $D$.
\end{lemma}

\begin{proof}

(i) implies (ii). This uses $\la D^\ast f,f\ra=\overline{\la Df,f\ra}$.

(ii) implies (iii) and (iv). We write $D=D_0+iD_1$ where $D_0=D_0^\ast\geq 0$ and $D_1=D_1^\ast$. The sectorial condition is equivalent to $-k_1 D_0\leq D_1\leq k_2 D_0$ where $k_r=\tan(\sig_r)$ for $r=1,\, 2$.
If (i) holds then $\la D_0f,f\ra=0$, so $\norm D_0^{1/2}f\norm^2=\la D_0f,f\ra=0$. This implies that $D_0^{1/2}f=0$, and hence that  $D_0f=0$. Since $0\leq D_1+kD_0\leq 2kD_0$, we also have $(\la D_1+kD_0)f,f\ra=0$, hence $(D_1+kD_0)^{1/2}f=0$ and then $(D_1+kD_0)f=0$. Therefore $D_1f=0$. We conclude that $Df=0$ and $D^\ast f=0$.

(iii) and (iv) separately imply (i). Both are elementary.

The property (iii) implies that $D\vert_\cK=0$. The property (iv) together with the general identity
\[
\overline{\Ran(D)}=(\Ker(D^\ast))^\perp
\]
implies that $\Ran(D)$ is dense in $\cK^\perp$. The corresponding statement for $D^\ast$ has a similar proof.

We have, finally, to prove (\ref{sectrep1}) and (\ref{sectrep2}). Without loss of generality we assume that $\cK=0$ and omit the symbol $\strut^\nat$. The operator $X=D_0$ is then one-one with dense range $\cD$ in $\cH$. The inequalities $-k_1 D_0\leq D_1\leq k_2 D_0$ are equivalent to $-k_1 I\leq E\leq k_2 I$ where $E=D_0^{-1/2}D_1D_0^{-1/2}$ is initially defined as a quadratic form on $\cD$. This yields (\ref{sectrep2}). The bounds on the form $E$ imply that it is associated with a bounded linear operator on $\cH$. We then have $D_1=D_0^{1/2}ED_0^{1/2}$ and hence (\ref{sectrep1}).
\end{proof}

\begin{corollary}\label{sectorialcor}
If $D$ is sectorial and $S$ is bounded then the following are equivalent.
\begin{enumerate}[(i)]
\item $SDS^\ast=0$;
\item $SD=0$;
\item $SD^\ast=0$.
\end{enumerate}
\end{corollary}

\begin{proof}
Assuming (i), $\la SDS^\ast g,g\ra=0$ for all $g\in \cH$. Therefore $\la Df,f\ra=0$ for all $f\in\Ran(S^\ast)$. Lemma~\ref{sectoriallemma} now implies that $Df=0$ for all $f\in \Ran(S^\ast)$. Hence $DS^\ast=0$ and (3) holds. The proof that (i) implies (ii) is similar and the proofs that (ii) and (iii) separately imply (i) are elementary.
\end{proof}

The remainder of this section is of independent interest, but it is not used elsewhere.
Given constants $\sig_1,\, \sig_2$ such that $-\pi/2< -\sig_1\leq 0 \leq \sig_2<\pi/2$, the set of all bounded operators $D$ on the Hilbert space $\cH$ such that (\ref{sectdef}) holds is a proper closed convex cone, which we denote by $\cC_{\sig_1,\sig_2}$. We say that a non-zero operator $C$ lies in $\partial \cC_{\sig_1,\sig_2}$ if $C=A+B$ and $A,\, B \in\cC_{\sig_1,\sig_2}$ imply that there exist non-negative constants $\alp,\, \bet$ such that $A=\alp C$ and $B=\bet C$. The set of all positive multiples of such an operator $C$ is called an extreme ray of $\cC_{\sig_1,\sig_2}$.

\begin{lemma}\label{kernelcone}
Let $A,\, B,\, C\in\cC_{\sig_1,\sig_2}$ and $C=A+B$. Then $\Ker (C)\subseteq \Ker (A)$. In particular $\rank(C)=1$ implies  $A=0$ or $\rank(A)=1$.
\end{lemma}

\begin{proof}
The assumptions imply that
\[
C+C^\ast=(A+A^\ast)+(B+B^\ast).
\]
and then
\[
0\leq A+A^\ast \leq C+C^\ast.
\]

Therefore $\la (C+C^\ast )f,f\ra=0$ implies $\la (A+A^\ast )f,f\ra=0$. Lemma~\ref{sectoriallemma} now implies that $\Ker (C)\subseteq \Ker (A)$.
\end{proof}

\begin{theorem}\cite{DS}\label{raytheorem}
Let $\cC_{\sig_1,\sig_2}$ be the cone defined above. Then a non-zero operator $A\in \partial \cC_{\sig_1,\sig_2}$ if and only if $Af=\alp \la f,e\ra e$ for all $f\in\cH$, where $e\in \cH$ satisfies $\norm e\norm \not=0$ and $\alp=\rme^{-i\sig_1}$ or $\alp=\rme^{i\sig_2}$.
\end{theorem}

\begin{proof}
Given $A\in \cC_{\sig_1,\sig_2}$, let $\cK_1=\Ker(A)$. If $\cK_1^\perp$ has dimension greater than $1$, then by applying the spectral theorem to the self-adjoint operator $E$ in (\ref{sectrep1}), one may write $\cK_1^\perp=\cK_2\oplus \cK_3$ where $\cK_2$ and $\cK_3$ are non-zero orthogonal subspaces that are invariant with respect to $E$. One then has a block decomposition
\[
I^\nat +iE= \mtrx{
I_2+iE_2&0\\ 0& I_3+iE_3}
\]
in an obvious notation. Moreover $I_2+iE_2$ and $I_3+iE_3$ both lie in $\cC_{\sig_1,\sig_2}$ with respect to the relevant Hilbert spaces. It follows that $A=A_2+A_3$ where $A_2$ and $A_3$ have the following block decompositions with respect to  $\cH=\cK_1\oplus\cK_2\oplus\cK_3$.
\begin{eqnarray*}
A_2&=&A_0^{1/2}\left(\begin{array}{ccc}
0&0&0\\
0&I_2+iE_2&0\\
0&0&0
\end{array}\right) A_0^{1/2},\\
A_3&=&A_0^{1/2}\left(\begin{array}{ccc}
0&0&0\\
0&0&0\\
0&0&I_3+iE_3
\end{array}\right) A_0^{1/2}.
\end{eqnarray*}
The factors $A_0^{1/2}$ do not change the sector in which the numerical range lies, so $A_2,\, A_3\in \cC_{\sig_1,\sig_2}$ and $A\notin \partial\cC_{\sig_1,\sig_2}$.

Conversely if $\cK_1^\perp$ is one-dimensional then $A$ has rank $1$ and it is of the form
$A f  =\la  f , e_1\ra e_2$ for some non-zero vectors $e_1,\, e_2$ and all $ f \in\cH$. Since
\begin{eqnarray*}
\Ker (A)&=& \{  f :\la  f ,e_1\ra=0\},\\
\Ker (A^\ast)&=&\{  f :\la  f ,e_2\ra=0\},
\end{eqnarray*}
Lemma~\ref{sectoriallemma} implies that $\Ker (A)=\Ker (A^\ast)$, from which one may deduce that $e_2$ is a multiple of $e_1$. An easy calculation using Lemma~\ref{kernelcone} shows that $A$ is in an extreme ray if and only if the argument of $\alp$ has one of the two stated values.
\end{proof}

\section{Cyclicity}\label{cyclicity}

This section generalizes the notion of cyclic vector to perturbations of an operator that have rank greater than $1$.

\begin{theorem}\label{cyclicequiv}
Let $A$ be a (possibly unbounded) self-adjoint operator acting in the Hilbert space $\cH$ and let $B,\, X$ be two bounded operators on $\cH$. Then the following conditions are equivalent.
\begin{enumerate}[(i)]
\item
$X\rme^{iAt}B=0$ for all $t\in \R$;
\item
$X(zI-A)^{-1}B=0$ for all $z\notin\Spec(A)$;
\item
$X\rme^{i(A+\gam B)t}B=0$ for some (equivalently all) $\gam\in \C$ and all $t\in\R$.

\hspace*{-2.5em}If $A$ is bounded the above conditions are also
equivalent to
\item
$XA^nB=0$ for all $n\geq 0$.
\end{enumerate}
\end{theorem}

\begin{proof}
We use a number of standard theorems and formulae from the theory of one-parameter semigroups; see \cite[Sections~8.2, 11.4]{LOTS}; in finite dimensions many of these can be derived more directly. We first observe that $A=A^\ast$ implies that there is a one-parameter group with generator $iA$; following the usual convention we write this in the form $\rme^{iAt}$, where $t\in\R$. The boundedness of $B$ implies that there is a one parameter group, which we denote by $\rme^{i(A+\gam B)t}$, whose generator is $A+\gam B$.\\
(i)$\implies$(ii). This follows directly from the following formulae, the integrals being convergent in the strong operator topology. If $\Im(z)<0$ then
\[
(zI-A)^{-1}= i\int_0^\infty \rme^{(-izI+iA)t} \, \rmd t.
\]
If $\Im(z)>0$ then
\[
(zI-A)^{-1}= -i\int_{-\infty}^0 \rme^{(-izI+iA)t} \, \rmd t.
\]
If $z\in\R\backslash \Spec(A)$ and $\eps >0$ then
\[
(zI-A)^{-1}= \lim_{\eps\to 0} (zI+i\eps I-A)^{-1}.
\]
(ii)$\implies$(i). This uses the formulae
\begin{eqnarray*}
(sI\mp iA)^{-n-1} &=& \frac{(-1)^n}{ n!}%
\frac{\rmd}{\rmd s^n}(sI\mp iA)^{-1},\\
\rme^{\pm iAt} &=& \lim_{n\to\infty} %
\left( \frac{t}{n}\right)^{-n}%
\left(\frac{n}{t}I\mp iA\right)^n.
\end{eqnarray*}
The formulae are valid for all positive $s$, $t$ and $n$ and the limits may be taken in the strong operator topology. Both formula may be proved by using the spectral theorem, but they are also valid at the semigroup level.\\
(i)$\implies$(iii). Assuming $t>0$, this uses the formula
\begin{eqnarray*}
\rme^{i(A+\gam B)t}&=&\rme^{iAt}+\int_{s=0}^t \rme^{iA(t-s)}i\gam B\rme^{iAs}\, \rmd s\\
&& +\int_{s=0}^t \int_{u=0}^s\rme^{iA(t-s)}i\gam
B\rme^{iA(s-u)}i\gam B\rme^{iAu}\, \rmd u\rmd s+\ldots ,
\end{eqnarray*}
the integrals and series being convergent in the strong operator topology for all $\gam\in\C$. The proof for $t<0$ is similar.\\
(iii)$\implies$(i). If (iii) holds for some $\gam\in\C$ then (i) follows by using the formula
\begin{eqnarray*}
\rme^{iAt}&=&\rme^{i(A+\gam B)t}-\int_{s=0}^t \rme^{i(A+\gam B)(t-s)}i\gam B\rme^{i(A+\gam B)s}\, \rmd s\\
&& +\int_{s=0}^t \int_{u=0}^s\rme^{i(A+\gam B)(t-s)}i\gam
B\rme^{i(A+\gam B)(s-u)}i\gam B\rme^{i(A+\gam B)u}\, \rmd u\rmd s+\ldots
\end{eqnarray*}
(i)$\Leftrightarrow$(iv). These use
\[
(iA)^n=\left. \frac{\rmd^n}{\rmd t^n} \rme^{iAt}\right|_{t=0},\hspace{3em}%
\rme^{iAt}=\sum_{n=0}^\infty \frac{(iAt)^n}{n!},
\]
both limits being in the operator norm.
\end{proof}

In the context of Theorem~\ref{cyclicequiv}, we say that the bounded
operator $B$ is \emph{cyclic} for $A$ if the conditions of the
following corollary hold.

\begin{corollary}\label{cycliccor}
Let $A$ be a possibly unbounded self-adjoint operator acting in the Hilbert space $\cH$ and let $B$ be a bounded operator on $\cH$. Then the following conditions are equivalent.
\begin{enumerate}[(i)]
\item
Whenever any of the equivalent conditions of
Theorem~\ref{cyclicequiv} holds for some bounded operator $X$ on
$\cH$, it follows that $X=0$.
\item
If one defines \[\cL_2=\lin\{ \rme^{iAt}Bv:t\in\R\mbox{ and }
v\in\cH\}\] then $\cL_2$ is dense in $\cH$.
\item
If one defines \[\cL_3=\lin\{ (sI-A)^{-1}Bv:s\notin\Spec(A)\mbox{ and }
v\in\cH\}\] then $\cL_3$ is dense in $\cH$.
\item
Assuming that $A$ is bounded, if one defines \[\cL_4=\lin\{
A^nBv:n=0,1,2,\ldots\mbox{ and } v\in\cH\}\] then $\cL_4$ is dense in $\cH$.
\end{enumerate}
\end{corollary}

\begin{proof}
(i)$\implies$(ii). If (ii) is false then the Hahn-Banach theorem implies that there exists a non-zero $\phi\in\cH$ such that $\la \phi, v\ra=0$ for all $v\in\cL_2$. If one defines $Xv=\la v,\phi\ra \phi$ for all $v\in\cH$ then one sees that $X\rme^{iAt}Bv=0$ for all $v\in\cH$ but $X\not= 0$, so Theorem~\ref{cyclicequiv}(i) is false.\\
(ii)$\implies$(i). If Theorem~\ref{cyclicequiv}(i) is false for some non-zero $X\in\cL(\cH)$ then $\cL_2\subseteq \Ker(X)\not= \cH$, so (ii) is false.\\
The proofs that (i)$\Leftrightarrow$(iii) and (i)$\Leftrightarrow$(iv) are very similar.
\end{proof}

In the following theorem and elsewhere we use the notations
$\C_+=\{z\in \C:\Im(z)>0\}$ and $\C_-=\{z\in \C:\Im(z)<0\}$. If $B$ is a sectorial operator with sectorial constants $\sig_1$, $\sig_2$ we define
\begin{equation}
S_B=\{0\}\cup \{ \gam\in\C :\gam\not= 0\mbox{ and } \sig_1<\arg(\gam)< \pi-\sig_2\}.\label{SBdef}
\end{equation}
The condition $\gam\in S_B$ implies that $\gam \la Bf,f\ra\in \C_+\cup\{ 0\}$ for all $f\in\cH$.

\begin{remark}
The conditions in Corollary~\ref{cycliccor} only depend on $B$ via the closure of its range $\cR_0=\overline{\{Bf: f\in\cH\}}$. In particular if $A$ and $B$ are both bounded, then $B$ is cyclic for $A$ if and only if the linear span of $\bigcup_{r\geq 0} A^r\cR_0$ is dense in $\cH$. Given $m\in\N$, let $\cR_m$ be the orthogonal complement of $\bigcup_{r= 0}^{m-1} A^r\cR_0$ in $\bigcup_{r= 0}^m A^r\cR_0$. Then $\cR_m$ are orthogonal subspaces and $B$ is cyclic for $A$ if and only if the closure of the sum of $\{ \cR_m\}_{m=0}^\infty$ is dense on $\cH$. One may use these subspaces to represent $A$ as a self-adjoint block tridiagonal matrix. If $B$ is sectorial and $\tB_{r,s}$ is its associated block matrix, then $\tB_{0,0}$ is the truncation of $B$ to $\cR_0$ and all other entries $\tB_{r,s}$ vanish. If $\cH$ is finite-dimensional, one only has a finite number of non-zero subspaces.
\end{remark}

\begin{remark}
The conditions in Corollary~\ref{cycliccor} are close to those used in the block Krylov subspace method of numerical analysis. Case~4 corresponds to the standard version of the method while Case~3 corresponds to the rational version.
\end{remark}

\begin{theorem}\label{SBimpliesC+}
Suppose that $B$ is sectorial and that $\gam\in S_B$. If $B$ is cyclic for $A$ and $\lam$ is an eigenvalue of $A_\gam$ then $\lam\in\C_+$. If $M=\rank (B)<\infty$ then the geometric multiplicity of $\lam$ is at most $M$.
\end{theorem}

\begin{proof}
Suppose that $0\not= f\in\Dom(A_\gam)$ and $Af+\gam Bf=\lam f$. By calculating the imaginary part of
\[
\la Af,f\ra+ \la\gam Bf,f\ra=\lam\la f,f\ra
\]
one deduces that either $\lam\in\C_+$ or $\lam\in\R$ and $\Im(\la\gam Bf,f\ra)=0$. Since $\gam B$ is sectorial it follows that $\la\gam Bf,f\ra=0$. Lemma~\ref{sectoriallemma} now implies that $Bf=B^\ast f=0$. Therefore $Af=\lam f$ and $\rme^{iAt}f=\rme^{i\lam t} f$ for all $t\in\R$. Therefore $B^\ast\rme^{iAt}f=0$ for all $t\in \R$ and
\[
\la f,\rme^{-iAt}Bv\ra=0
\]
for all $t\in\R$ and all $v\in\cH$. Since $B$ is cyclic for $A$ it
follows by Corollary~\ref{cycliccor}(ii) that $f=0$. The contradiction implies that $\lam\in\C_+$.

If $(A+\gam B)f=\lam f$ then $(\lam I-A)f=\gam Bf$. Since $\lam\in \C_+$, $\lam\notin\Spec(A)$ and $f=(\lam I-A)^{-1}\gam Bf\in (\lam I-A)^{-1} B\cH$, which is a linear subspace with dimension at most $M$.
\end{proof}

\section{The main theorems when $N<\infty$}\label{maintheorems}

In this section we suppose that $N=\dim(\cH)<\infty$ and put
$M=\rank (B)$ where $B$ is sectorial. Our goal is to describe how the spectrum of $A_\gam=A+\gam B$ depends on $\gam$, assuming that $\gam\in S_B$ as defined in (\ref{SBdef}), and in particular the relationship between the spectral asymptotics for small and for large $\gam$.

Under the above assumptions it is elementary that $\Im(\la (A+t\gam B)f,f\ra)$ is a monotonically increasing linear function of $t\in (0,\infty)$, as is $\Im(\tr((A+t\gam B)))$. Combining these observations with known variational results for $B=B^\ast\geq 0$ and $\gam>0$, leads to the conjecture that the imaginary part of each eigenvalue of $A+t\gam B$ also increases monotonically as a function of $t$. The following example demonstrates that this is false. It also illustrates the results in Theorem~\ref{maintheorem}. Example~1.5.7 of \cite{LOTS}, which is even simpler, provided one of the motivations for the present study.

\begin{example}\label{rank2example}
Let $A$ be the $5\times 5$ diagonal matrix with eigenvalues
$\lam_r=r$ for $1\leq r\leq 5$, and let $A_\gam=A+\gam B$ where $B$ is
the rank $2$ operator
\[
Bf=\la f ,e_1\ra e_1 +\la f,e_2\ra e_2
\]
for all $f\in \C^5$, where $e_1=(2,2,2,2,2)$ and $e_2=(3,3,-2,-2,-2)$. Figure~\ref{rank2fig} plots the eigenvalues of $A_\gam$ for $\gam=t\rme^{i\theta}$, where $0<t<\infty$ and $\theta=3\pi/8$. The eigenvalues converge to the eigenvalues of $A$ as $t\to 0$. Two of the eigenvalue curves diverge as $t\to\infty$, while the other three converge back to the real axis.\end{example}

\begin{figure}[h!]
\begin{center}
\draft{\resizebox{15cm}{!}{\mbox{\rotatebox{0}%
{\includegraphics{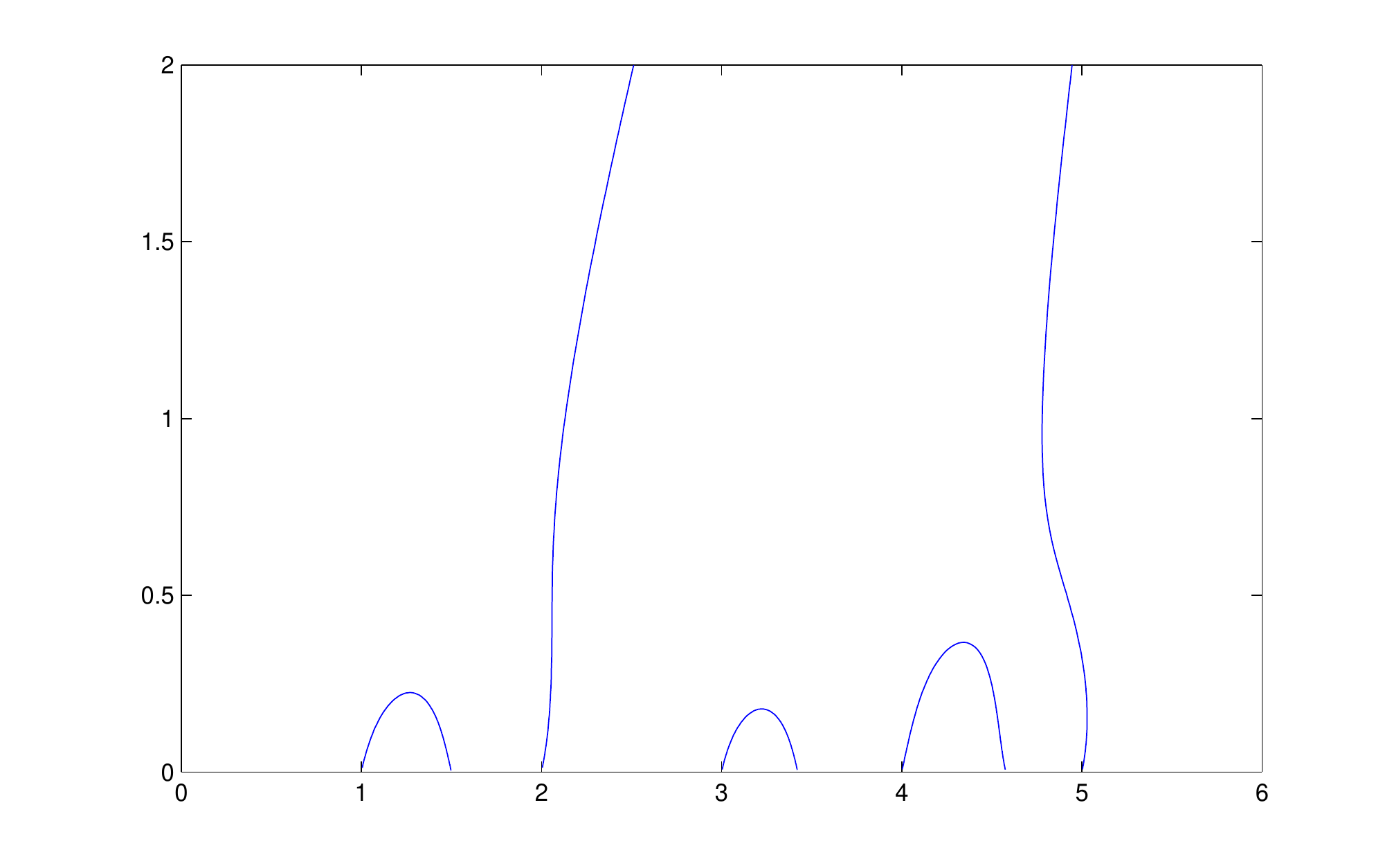}}}}}
\end{center}
\caption{Spectral curves described in %
Example~\ref{rank2example}}\label{rank2fig}
\end{figure}

We shall need the following
conditions. Apart from (H1), each is generic in the sense that it holds for a dense open subset of operators of the relevant type.
\begin{description}
\item[(H1)] $\dim(\cH)<\infty$, $A=A^\ast$ and $B$ is sectorial.
\item[(H2)] The operator $B$ is cyclic for the operator $A$.
\item[(H3)] All of the eigenvalues of $A$ have algebraic multiplicity $1$.
\item[(H4)] All of the non-zero eigenvalues of $B$ have algebraic multiplicity $1$.
\item[(H5)] All of the eigenvalues of the truncation of $A$ to the kernel $\cK$ of $B$ have algebraic multiplicity $1$.
\end{description}

\begin{theorem}\label{contractions}
Let $\gam\in S_B$. If (H1) holds and $Z=i(A+\gam B)$ then $\norm \rme^{Zt}\norm\leq 1$ for all $t\geq 0$. Given (H1), the condition (H2) holds if and only if there are constants $M\geq 1$ and $c>0$ such that
\begin{equation}
\norm \rme^{Zt}\norm \leq M\rme^{-ct}\label{strictcontr}
\end{equation}
for all $t\geq 0$. Given (H1) and (H2), one can put $M=1$ in (\ref{strictcontr}) if and only if $\Ker(B)=\{ 0\}$.
\end{theorem}

\begin{proof}
It follows directly from (H1) that $Z$ is dissipative for every $\gam\in S_B$ and hence that $\rme^{Zt}$ is a contraction semigroup for $t\geq 0$. If (H2) also holds then every eigenvalue $\lam$ of $Z$ satisfies $\Re(\lam )<0$, and an application of the Jordan form theorem yields (\ref{strictcontr}).
\end{proof}

\begin{example} Suppose that $A$ and $B$ satisfy (H1--5) and that every eigenvalue $\lam$ of $A$ satisfies $\lam>0$. Define the operators $\tA$ and $\tB$ on $\cH\oplus \cH$ by $\tA(f\oplus g)=(Af)\oplus (cAg)$ and $\tB(f\oplus g)=(Af)\oplus (cBg)$, where $c>0$. Then $\tA$ and $\tB$ satisfy (H1--H5) for almost all such $c$, but not for $c=1$. The proof uses Lemma~\ref{trunkspec}.
\end{example}

We assume (H1), (H2) and that $\gam\in S_B$ throughout the section, so that we can use Theorem~\ref{SBimpliesC+}. We make constant use of the polynomial
\begin{eq}
p(\gam,\lam)=\det(A+\gam B-\lam I).\label{pgl}
\end{eq}
The $\gam$-dependence of the spectrum of $A_\gam$ depends on an
analysis of the algebraic surface
\begin{eq}
\cR=\{ (\gam,\lam)\in S_B\times\C_+:p(\gam,\lam)=0\} .
\end{eq}

We will use the following classical facts.

\begin{proposition}\label{polyroots}
If $X$ is an $N\times N$ matrix and $q(\lam)=\det(X-\lam I)$ then
$q$ is a polynomial with degree $N$ and the following are
equivalent.\vspace{-2ex}
\begin{enumerate}[(i)]
\item Every eigenvalue of $X$ has algebraic multiplicity $1$;\vspace{-1.5ex}
\item Every root $\lam$ of $q$ is simple;\vspace{-1.5ex}
\item There are no simultaneous solutions of $q(\lam)=q^\pr(\lam)=0$;\vspace{-1.5ex}
\item The discriminant of $q$ is non-zero. (The discriminant of a polynomial $q$ is a certain multiple of the square of its Vandermonde determinant, and may be written as a homogeneous polynomial with degree $2N-2$ in the coefficients of $q$.)
\end{enumerate}
\end{proposition}

Since the zeros of $p(0,\lam)$ all lie on the real axis, the following lemma can often be used to reduce the determination of the zeros of $p(\gam,\lam)$ in $\C_+^2$ to a lower dimensional problem. See Lemma~\ref{BS}. The right-hand side of (\ref{reldeteq}), usually without the $\str^\nat$, is called the relative determinant of $A_\gam$ and $A$.

\begin{lemma}\label{relativedet}
One has
\begin{equation}
\frac{p(\gam,\lam)}{p(0,\lam)}=\det \left( (I+\gam(A-\lam I)^{-1}B)^\nat\right)\label{reldeteq}
\end{equation}
where $\str^\nat$ denotes the truncation of the operator to the range of $B^\ast$.
\end{lemma}

\begin{proof}
This is a combination of two identities
\begin{eqnarray*}
\frac{p(\gam,\lam)}{p(0,\lam)}&=&\det  (I+\gam(A-\lam I)^{-1}B),\\
&=&\det \left( (I+\gam(A-\lam I)^{-1}B)^\nat\right).
\end{eqnarray*}
The first equality is obtained by calculating the determinants of both sides of the identity
\[
A+\gam B -\lam I=(A-\lam I)(I+\gam(A-\lam I)^{-1}B).
\]
The second equality is proved by writing $I+\gam(A-\lam I)^{-1}B$ as a $2\times 2$ block matrix using the orthogonal decomposition
\[
\cH=\Ker(B)\oplus \Ran(B^\ast).
\]
\end{proof}

\begin{lemma}\label{useH2}
Given (H2) and (H3), there exists a finite set $F_1\subset S_B$, such that $A_\gam$ has $N$ distinct eigenvalues, each with algebraic multiplicity $1$, for every $\gam\in S_B \backslash F_1$. If $(\gam,\lam)\in \cR$ and $\gam\notin F_1$ then $\frac{\partial p}{\partial \lam}(\gam,\lam)\not= 0$.
\end{lemma}

\begin{proof}
The eigenvalues of $A_\gam$ are the roots of the polynomial
$q_\gam(\lam)=p(\gam,\lam)$, which is of degree $N$ in $\lam$ with
leading coefficient $(-1)^N$. The eigenvalues of $A_\gam$ lie in $\C_+$ by Theorem~\ref{SBimpliesC+}. They all have
algebraic multiplicity $1$ if and only if the discriminant of
$q_\gam$ is non-zero, by Proposition~\ref{polyroots}. The
coefficients of $q_\gam$ are polynomials in $\gam$, so the
discriminant is also a polynomial $r$ in $\gam$. The hypothesis (H3)
implies that $r(0)\not= 0$, so $r$ is not identically zero, and it
has only a finite number of roots. The first part of the proof is
completed by putting $F_1=\{\gam\in S_B:r(\gam)=0\}$. The proof of
the final part of the theorem uses Proposition~\ref{polyroots}
again.
\end{proof}

\begin{lemma}\label{useH3}
Given (H2) and (H4), there exists a finite set $F_2\subset S_B$, such that if $(\gam,\lam)\in \cR$ and $\gam\notin F_2$ then $\frac{\partial
p}{\partial \gam}(\gam,\lam)\not= 0$.
\end{lemma}

\begin{proof}
One may evaluate $p(\gam,\lam)$ by combining an orthonormal basis of $\Ker (B)$ with a set of $M$ eigenvectors associated with the
non-zero eigenvalues $\bet_1,\ldots,\bet_M$ of $B$. If one does so
then one sees that $q_\lam(\gam)=p(\gam,\lam)$ is a polynomial with
degree (at most) $M$ in $\gam$ whose leading coefficient is $\det(A^\nat-\lam I^\nat)\prod_{r=1}^M\bet_r$, where $A^\nat$ is the truncation of $A$ to $\Ker(B)$ and $I^\nat$ is the identity operator on this subspace. Since $A^\nat$ is self-adjoint and $\lam\in\C_+$, the determinant is non-zero and the degree of $q_\lam$ is $M$.

One may see as in the proof of Lemma~\ref{useH2} that the roots of
$q_\lam$ are all distinct if and only if a certain polynomial
$r(\lam)$ is non-zero. If $p(\gam,\lam)=0$ and $\frac{\partial
p}{\partial \gam}(\gam,\lam)=0$ then $r(\lam)=0$.  The set $G_1$ of
roots of $r$ is finite provided $r$ does not vanish identically.
Assuming this,
\[
F_2=\{ \gam\in \C_+:(\gam,\lam)\in \cR \mbox{ for some } \lam\in G_1\}
\]
is also finite and $\frac{\partial p}{\partial
\gam}(\gam,\lam)\not=0$ for all $(\gam,\lam)\in \cR$ such that
$\gam\notin F_2$.

The polynomial $r$ is not identically zero provided the $M$
solutions $\gam$ of $\det(A+\gam B-\lam I)=0$ are distinct for all
large enough $\lam\in\C_+$. This is true if and only if the
solutions $s$ of $\lam^{-N}\det(A+s\lam B-\lam I)=0$ are distinct
for all large enough $\lam\in\C_+$. These solutions converge as
$|\lam|\to\infty$ to the solutions of $\det(sB-I)=0$, which are
$\bet_1^{-1},\ldots, \bet_M^{-1}$. They are distinct by (H4).
\end{proof}

The following lemma will be used in the proof of
Theorem~\ref{maintheorem}.

\begin{lemma}\label{inverses}
Let $L$ be a bounded operator on $\cH=\cH_1\oplus \cH_2$ with block
matrix
\[
L=\mtrx{P&Q \\ R&S},
\]
where the entries satisfy $\norm Q\norm \leq c$, $\norm R\norm \leq
c$, $\norm S^{-1}\norm \leq 1/(2c)$ and $\norm P^{-1}\norm <\eps\leq
1/(2c)$. Then $L$ is invertible and
\[
\left\| L^{-1}-\mtrx{0&0\\ 0&S^{-1}}\right\| <2\eps.
\]
\end{lemma}

\begin{proof}
If one puts $X=\mtrx{P&0 \\ 0&S}$ and $Y=\mtrx{0&Q \\ R&0}$ then
$\norm Y\norm \leq c$ and $\norm X^{-1}\norm \leq 1/(2c)$. Therefore
$\norm YX^{-1}\norm\leq 1/2$ and the perturbation expansion
\[
(X+Y)^{-1}=X^{-1}\sum_{n=0}^\infty (-YX^{-1})^n
\]
implies that $L=X+Y$ is invertible with $\norm L^{-1}\norm \leq
1/c$. Moreover
\begin{eqnarray*}
\left\| L^{-1}-\mtrx{0&0\\ 0&S^{-1}}\right\|  &\leq&%
\norm (X+Y)^{-1}-X^{-1}\norm +\eps\\
&\leq &\norm
-X^{-1}YX^{-1}+(X^{-1}YX^{-1})(YX^{-1})\\
&&-(X^{-1}YX^{-1})(YX^{-1})^2+\ldots\norm +\eps\\
&\leq & 2\norm X^{-1}YX^{-1}\norm  +\eps\\
&=&2\left\Vert\mtrx{0&P^{-1}Q S^{-1}\\
S^{-1}RP^{-1}&0}\right\Vert +\eps\\
&\leq 2\eps.
\end{eqnarray*}

\end{proof}

We use the above results to connect the spectrum of $A_\gam$ for large and small $\gam$. Let $\cG$ denote the set of all continuously differentiable curves $g:[0,\infty)\to \C$ such that $g(0)=0$, $g(t)\in S_B$ for every $t>0$ and $g^\pr(t)$ does not vanish for any $t\in [0,\infty)$. Let $F=F_1\cup F_2$ where $F_1$ is defined as in Lemma~\ref{useH2} and $F_2$ is defined as in Lemma~\ref{useH3}. Let $\cG_0$ denote the set of all curves $g\in\cG$ such that $g(0)=0$, $g(t)\in S_B\backslash F$ for every $t>0$ and $\lim_{t\to\infty}|g(t)|=\infty$. In the next theorem, one can impose stronger conditions on $g$ (e.g.\ $C^\infty$ or real analyticity) and obtain similarly strengthened conclusions on the eigenvalue curves $\lam_r$.

Our main theorem below is an example of monodromy in the sense that we prove that certain one-parameter curves that avoid a finite number of singularities may have different end points even if they have the same starting point, provided they take different routes around the singularities; the difference is measured by an element of a permutation group.

\begin{theorem}\label{maintheorem}
Given (H2--5), let $g\in\cG_0$. Then there exist $N$ curves
$\lam_r\in\cG$ such that $\Spec(A_{g(t)})=\{\lam_1(t),\ldots,
\lam_N(t)\}$ for all $t\in[0,\infty)$. One can choose the ordering
of these so that $\lam_r(0)=\alp_r$ for all $r\in\{1,\ldots, N\}$,
where $\alp_r$ are the eigenvalues of $A$ written in increasing
order. Assuming this is done, there exists a $g$-dependent
permutation $\pi$ on $\{1,\ldots, N\}$ such that
\[
\lim_{t\to\infty}\frac{\lam_{\pi(r)}(t)}{g(t)}=\bet_r
\]
for $1\leq r\leq M$, where $\bet_r$ are the non-zero eigenvalues of
$B$ written in any fixed order, and
\[
\lim_{t\to\infty}\lam_{\pi(M+r)}(t)=\del_r
\]
for $1\leq r\leq N-M$, where $\del_r$ are the non-zero eigenvalues
of the truncation $A^\nat$ of $A$ to $\Ker(B)$ written in increasing
order (the eigenvalues are distinct by (H5). If $g\in \cG_0$ is a real analytic curve then so are all the curves $\lam_r$.
\end{theorem}

\begin{proof}
If $t\geq 0$ then $g(t)\notin F$, so the eigenvalues $\lam_1,\ldots, \lam_N$ of $H_{g(t)}$ all have algebraic multiplicity $1$. Perturbation theory implies that each eigenvalue of $A_\gam$ is an analytic function of $\gam$ if $\gam=g(t)$. Therefore each eigenvalue $\lam(t)$ of $A_{g(t)}$ is a $C^1$ function of $t$, or real-analytic if $g$ is real analytic. These perturbation arguments imply all the statements of the theorem that relate to the limit $t\to 0$.

We next observe that $p(g(t),\lam(t))=0$ for all $t> 0$. Differentiating this with respect to $t$ yields
\[
\frac{\partial p}{\partial \gam}(g(t),\lam(t))g^\pr(t)+\\
\frac{\partial p}{\partial \lam}(g(t),\lam(t))\lam^\pr(t)=0.
\]
By applying Lemmas~\ref{useH2} and \ref{useH3}, we deduce that $\lam^\pr(t)$ is non-zero for every $t>0$.

In order to prove the remainder of the theorem we need only find the asymptotic forms of the eigenvalues of $H_\gam$ as $|\gam|\to\infty$ and apply the results to $\gam=g(t)$ as $t\to\infty$. The spectrum of $A_\gam$ is a set rather than an ordered sequence and there is no reason for any ordering of the eigenvalues of $A_\gam$ for large $\gam$ to be related to the ordering for $\gam = 0$.

We start by describing the large eigenvalues of $A_\gam$. For every $r\in \{ 1,\ldots, M\} $, perturbation theory and (H4) together imply that $B+\gam^{-1}A$ has a simple eigenvalue of the form
\[
\mu_r=\bet_r+\gam^{-1}\la A f_r,f_r^\ast\ra+O(\gam^{-2})
\]
as $|\gam|\to\infty$, where $f_r$ is an eigenvector of $B$ associated with the eigenvalue $\bet_r$, $f_r^\ast$ is an eigenvector of $B^\ast$ associated with the eigenvalue $\overline{\bet_r}$ and we normalize so that $\la f_r,f_r^\ast\ra=1$. This implies that $A_\gam$ has a simple eigenvalue of the form
\[
\lam_r=\gam\bet_r+\la A f_r,f_r^\ast\ra+O(\gam^{-1})
\]
for all $r\in\{1,\ldots,M\}$.

We next use Lemma~\ref{inverses} to describe the small eigenvalues of $A_\gam$. If one defines $\cH_1=\Ran(B)$ and $\cH_2=\Ker(B)$ then $\C^N=\cH_1\oplus \cH_2$ is an orthogonal direct sum by Lemma~\ref{sectoriallemma}. One may write
\[
A+\gam B=\mtrx{C+\gam B^\nat &E\\ E^\ast & A^\nat}.
\]
where $B^\nat$ is the truncation of $B$ to $\cH_1$, $C$ is the truncation of $A$ to $\cH_1$ and $A^\nat$ is the truncation of $A$ to $\cH_2$. We now add $kI$ to both sides where the constant $k$ is independent of $\gam$ and large enough to ensure that $\norm (A^\nat+kI)^{-1}\norm\leq 1/(2c)$, where $c=\norm E\norm +1=\norm E^\ast\norm+1$. Using the fact that $B^\nat$ is invertible on $\cH_1$, we observe that
\[
\eps=\norm(C+\gam B^\nat+kI)^{-1}\norm=O(|\gam|^{-1})
\]
as $|\gam|\to\infty$. Lemma~\ref{inverses} now implies that $A+\gam B+kI$ is invertible and
\begin{eq}
\left\Vert (A+\gam B+kI)^{-1}-\mtrx{0&0\\0&(A^\nat+kI)^{-1}}\right\Vert%
=O(|\gam|^{-1})\label{ABDcompare}
\end{eq}
for all large enough $|\gam|$. Every eigenvalue $\del_r$ of $A^\nat$ satisfies
\[
2\leq 2c\leq |\del_r+k|\leq \norm A^\nat\norm +k.
\]
Since $A^\nat$ is self-adjoint, a perturbation argument applied to (\ref{ABDcompare}) implies that there is an eigenvalue $\mu_r$ of $A+\gam B$ such that
\[
|(\mu_r+k)^{-1}-(\del_r+k)^{-1}|=O(|\gam|^{-1})
\]
as $|\gam|\to\infty$. By combining the last two equations we obtain
\[
|\mu_r-\del_r|=O(|\gam|^{-1})
\]
as $|\gam|\to\infty$. Moreover, the perturbation argument proves that $\mu_r$ has the same multiplicity $1$ as $\del_r$ for all $r\in\{1,\ldots,N-M\}$.

We have now described $N$ distinct simple eigenvalues of $A_\gam$ for all sufficiently large $|\gam|$. Since $A_\gam$ is an $N\times N$ matrix there are no other eigenvalues.
\end{proof}

Simple continuity arguments show that two homotopic curves $g_1,\, g_2\in \cG_0$ give rise to the same permutation $\pi$. The fact that non-homotopic curves may give rise to different permutations is demonstrated in Examples~\ref{rank1example} and \ref{Neq5example}.

\section{Localization}

In this section we describe a procedure for approximating the spectrum of $A_\gam=A+\gam B$ in a given region of $\C_+$. We assume that $A$ is a (possibly unbounded) self-adjoint operator on $\cH$, that $\cK$ is an auxiliary Hilbert space, that $B=CD$ and that $C:\cK\to \cH$, $D:\cH\to\cK$ are bounded operators.

We first note that the Birman-Schwinger method does not depend on self-adjointness of the perturbation. Numerically, the method is most useful when the dimension of $\cK$ is much smaller than that of $\cH$, but one need not assume that either is finite-dimensional.

\begin{lemma}\label{BS}
If $\lam\notin \Spec(A)$ and $\gam\not= 0$ then $\lam$ is an eigenvalue of $A_\gam$ if and only if $-1/\gam$ is an eigenvalue of
\[
m(\lam)=D(A-\lam I)^{-1}C \in \cL(\cK).
\]
 If $\cK$ is finite-dimensional then $\lam\notin \Spec(A)$ is an eigenvalue of $A_\gam$ if and only if the jointly analytic function
\[
p(\gam,\lam)=\det\left(I+\gam m(\lam) \right)
\]
vanishes.
\end{lemma}

\begin{proof}
We start with the identity
\[
A+\gam CD-\lam I= (I+\gam CD(A-\lam I)^{-1})(A-\lam I),
\]
both sides being regarded as linear maps from $\Dom(A)$ to $\cH$. Since $A-\lam I:\Dom(A)\to \cH$ is one-one and onto, $\lam$ is an eigenvalue of $A+\gam CD$ if and only if $-1/\gam$ is an eigenvalue of $CD(A-\lam I)^{-1}$. Both implications in the first sentence now depend on the elementary fact that if $U,\, V$ are vector spaces over $\C$, $X:U\to V$, $Y:V\to U$ are linear operators and $\sig\in\C$ is non-zero, then $\sig$ is an eigenvalue of $XY$ if and only if it is an eigenvalue of $YX$.
\end{proof}

In spite of the second statement in Lemma~\ref{BS}, the $\cL(\cK)$-valued function $m$ is easier to analyze than the scalar function $p$. One says that the analytic function $m:\C_+\to \cL(\cK)$ is an operator-valued Herglotz function if $\la m(\lam)f,f\ra \in \C_+$ for every $f\in \cK\backslash \{ 0\}$ and $\lam\in\C_+$; see \cite{GKMT}.

\begin{lemma}\label{mHerglotz}
Suppose that $B=B^\ast\geq 0$, $C=D=B^{1/2}$,  $\cK$ is the closure of the range of $B$ and $\str^\nat$ is the operation of truncation to $\cK$. Then
\begin{equation}
m(\lam)=\left( B^{1/2}(A-\lam I)^{-1}B^{1/2}\right)^\nat\label{mHdef}
\end{equation}
is a $\cL(\cK)$-valued Herglotz function.
\end{lemma}

\begin{proof}
The assumptions imply that $g=B^{1/2}f\not= 0$ and that
\[
\la m(\lam)f,f\ra=\la (A-\lam I)^{-1}g,g\ra ,
\]
which lies in $\C_+$ by the spectral theorem.
\end{proof}

\begin{theorem}
If $B=B^\ast\geq 0$ has finite rank $N$ and $\lam\in\C_+$ then there are at least $1$ and at most $N$ values of $\gam$ such that $\lam$ is an eigenvalue of $A_\gam$; all such $\gam$ lie in $\C_+$.
\end{theorem}

\begin{proof}
If $f\in \Dom(A)$, $f\not=0$ and $A_\gam f=\lam f$ then
\[
\la Af,f\ra+\gam \la Bf,f\ra=\lam \la f,f\ra.
\]
This implies that
\[
\Im(\gam) \la Bf,f\ra=\Im(\lam)\la f,f\ra.
\]
Since the right hand side is positive, we deduce that $\la Bf,f\ra>0$ and $\Im(\gam)>0$.

Lemma~\ref{BS} states that $\lam$ is an eigenvalue of $A_\gam$ if and only if $-1/\gam$ is an eigenvalue of the $N\times N$ matrix $m(\lam)$. The final statement of Lemma~\ref{mHerglotz} implies that every eigenvalue of $m(\lam)$ lies in $\C_+$. This proves that there are at least $1$ and at most $N$ distinct values of $\gam$, each of which lies in $\C_+$.
\end{proof}

The equation (\ref{mQformula}) below is a special case of the Nevanlinna-Riesz-Herglotz representation of operator-valued Herglotz functions; see \cite{GKMT}.

\begin{lemma}\label{Hergrep}
Under the assumptions of Lemma~\ref{mHerglotz}, let $P(E)$ denote the spectral projection of $A$ associated with any Borel subset $E\subseteq \R$. If
\[
Q(E)=\left( B^{1/2}P(E)B^{1/2}\right)^\nat
\]
then $Q$ is a finite, non-negative, countably additive, $\cL(\cK)$-valued measure on $\R$ and
\begin{equation}
m(\lam)=\int_\R \frac{1}{s-\lam} Q(\rmd s)\label{mQformula}
\end{equation}
for all $\lam\in \C_+$.
\end{lemma}

\begin{proof}
The formula (\ref{mQformula}) follows directly from
\[
(A-\lam I)^{-1}=\int_\R \frac{1}{s-\lam} P(\rmd s),
\]
which is proved using the spectral theorem.
\end{proof}

The following lemma is well-known, but we include a proof for completeness.

\begin{lemma}\label{Qintbound}
If $f:[a,b]\to\C$ is bounded and measurable then
\[
\norm \int_a^b f(s)\, Q(\rmd s)\norm \leq \norm f\norm_\infty\norm Q([a,b])\norm.
\]
\end{lemma}

\begin{proof}
If $\phi,\, \psi \in \cK$, we have
\begin{eqnarray*}
|\la \int_a^b f(s)\,  Q(\rmd s)\phi, \psi\ra |
&=& |\la \int_a^b f(s)\,  P(\rmd s)(B^{1/2}\phi), (B^{1/2}\psi)\ra |\\
&=&|\la f(A)P([a,b])(B^{1/2}\phi), %
(B^{1/2}\psi)\ra |\\
&=&|\la f(A)(P([a,b])B^{1/2}\phi), %
(P([a,b])B^{1/2}\psi)\ra |\\
&\leq & \norm f\norm_\infty\norm P([a,b])B^{1/2}\phi\norm\,\norm P([a,b])B^{1/2}\psi\norm\\
&=& \norm f\norm_\infty %
\la Q([a,b])\phi,\phi\ra^{1/2} %
\la Q([a,b])\psi,\psi\ra^{1/2}\\
&\leq & \norm f\norm_\infty %
\norm Q([a,b])\norm\, \norm \phi\norm\, \norm \psi\norm.
\end{eqnarray*}
The lemma follows.
\end{proof}

Our next lemma defines two operators that will be used in Theorem~\ref{trunc2}.

\begin{lemma}\label{XYdef}
Let $[a,b]\subset \R$. Then there exist bounded, self-adjoint operators $X,\, Y:\cK\to\cK$ such that
\begin{eqnarray}
X&=&\int_a^b Q(\rmd s), \label{Xdef}\\
X^{1/2}YX^{1/2} &=& \int_a^b s Q(\rmd s).\label{Ydef}
\end{eqnarray}
Moreover $0\leq X\leq B^\nat$ and $aI\leq Y\leq bI$.
\end{lemma}

\begin{proof}
Since $X=Q([a,b])$, the inequalities for $X$ and the boundedness of $X$ follow from
\[
0\leq \la Q([a,b])f,f\ra \leq \la Q(\R)f,f\ra=
\la B^\nat f,f\ra\leq\norm B^\nat\norm \, \norm f\norm^2,
\]
valid for all $f\in\cK$. If one defines
\[
Z=\int_a^b s Q(\rmd s)
\]
then $aX \leq Z\leq bX$. If $X$ is invertible this immediately implies that
\[
aI\leq Y=X^{-1/2} ZX^{-1/2}\leq bI.
\]
The general case follows as in Lemma~\ref{sectoriallemma}.
\end{proof}

Determining the eigenvalues of $m(\lam)$ for a given range of values of $\lam$ can sometimes be aided by writing
\[
m(\lam)=m_1(\lam)+m_2(\lam)
\]
where $m_1$ may be computed more readily than $m$
and $m_2(\lam)$ can be neglected or replaced by an appropriate approximation for the selected range of values of $\lam$. Theorem~\ref{trunc2} enables one to replace the contribution of an interval $[a,b]$ to $m(\lam)$ by a single operator provided $\lam$ is far enough away from $[a,b]$.

\begin{theorem}\label{trunc2}
Let
\[
m(\lam)=\int_\R \frac{1}{s-\lam}\, Q(\rmd s),
\]
where $\lam\in\C_+$ and $Q$ is a finite, non-negative, countably additive, $\cL(\cK)$-valued measure on $\R$. Given $[a,b]\subset \R$ and $L>0$, let
\begin{equation}
\tilde{m}(\lam)=\int_{s\notin[a,b] } \frac{1}{s-\lam}\, Q(\rmd s)+X^{1/2}(Y-\lam I)^{-1}X^{1/2},\label{mtildedef}
\end{equation}
where $X$ and $Y$ are as defined in Lemma~\ref{XYdef}. Then
\[
|m(\lam)-\tilde{m}(\lam)|\leq %
\frac{2(b-a)^2}{L^3}\norm Q([a,b])\norm
\]
for all $\lam\in \C_+$ such that $\dist(\lam, [a,b])\geq L$.
\end{theorem}

\begin{proof}
It suffices to prove that
\begin{equation}
\norm \int_a^b \frac{1}{s-\lam}\, Q(\rmd s)-X^{1/2}(Y-\lam I)^{-1}X^{1/2}\norm%
\leq  \frac{2(b-a)^2}{L^3}\norm Q([a,b])\norm \label{correct0}
\end{equation}
for all $\lam$ satisfying the stated conditions.

We first observe that
\begin{equation}
\frac{1}{s-\lam}+\frac{1}{\lam-b}%
+\frac{s-b}{(\lam-b)^2}=%
\frac{(s-b)^2}{(s-\lam)(\lam-b)^2}.\label{correct1}
\end{equation}
Integrating both sides with respect to $Q$ over $[a,b]$ yields
\begin{eqnarray}
\lefteqn{\hspace*{-4em} \int_a^b\frac{1}{s-\lam}\, Q(\rmd s)+\frac{X}{\lam-b}%
+\frac{X^{1/2}YX^{1/2}-bX}{(\lam-b)^2}  }\nonumber\\
&=& \int_a^b\frac{(s-b)^2}{(s-\lam)(\lam-b)^2}\, Q(\rmd s),\label{correct2}
\end{eqnarray}
and then
\begin{eqnarray}
\lefteqn{\hspace*{-4em}\norm\int_a^b\frac{1}{s-\lam}\, Q(\rmd s)+\frac{X}{\lam-b}%
+X^{1/2}\frac{Y-bI}{(\lam-b)^2}X^{1/2}\norm}\nonumber\\
&\leq & \frac{(b-a)^2}{L^3} \norm Q([a,b])\norm\label{correct3}
\end{eqnarray}
by Lemma~\ref{Qintbound}.

We next use the formula
\begin{equation}
\frac{I}{Y-\lam I}+\frac{I}{\lam-b}%
+\frac{Y-bI}{(\lam-b)^2}=%
\frac{(Y-bI)^2}{(Y-\lam I)(\lam-b)^2}
\label{correct4}
\end{equation}
to obtain
\begin{eqnarray}
\lefteqn{\hspace*{-4em}\norm X^{1/2}(Y-\lam I)^{-1}X^{1/2}+\frac{X}{\lam-b}%
+X^{1/2}\frac{Y-bI}{(\lam-b)^2}X^{1/2}\norm }\nonumber \\
&=&%
\norm X^{1/2}\frac{(Y-bI)^2}{(Y-\lam I)(\lam-b)^2}X^{1/2}\norm\nonumber \\
&\leq &\frac{(b-a)^2}{L^3}\norm Q([a,b])\norm.
\label{correct5}
\end{eqnarray}
The proof of (\ref{correct0}) is completed by combining (\ref{correct3}) and (\ref{correct5}).
\end{proof}

\begin{remark}
One can obtain a better approximation than that in Theorem~\ref{trunc2} if $[a,b]$ is divided into several subintervals, each of which is used to produce an extra term in the formula (\ref{mtildedef}). The new $\tilde{m}$ is, of course, more cumbersome to use numerically.
\end{remark}

From this point onwards we assume that $A$ is a possibly unbounded self-adjoint operator and that $A_\gam=A+\gam B$ where $Bf=\la f,e\ra e$ and $e\in\cH$ is a vector with norm $1$. The Herglotz function (\ref{mHdef}) is then scalar-valued and given by the formula %
$m(\lam)=\la (A-\lam I)^{-1}e,e\ra$.

If one approximates $m$ uniformly in a chosen region by another analytic function whose zeros are more easily computed, then one can apply Rouch\'{e}'s theorem to approximate the zeros of $m$ and hence the spectrum of $A_\gam$. Lemma~\ref{rouche} is directly applicable to the polynomial $p(\gam,\lam)=\det(A_\gam-\lam I)$, where $\gam$ is fixed. The connection between this and the Herglotz function $m$ is explained in Lemma~\ref{relativedet} and (\ref{rank1det}). The proof of Lemma~\ref{rouche} can be adapted to cases in which $p$ is not a polynomial; one needs an upper bound on the orders of its zeros, a lower bound on the distances between zeros and a lower bound on $p$ for points that are not close to a zero.

\begin{lemma}\label{rouche}
Let $0<\eps<1/2$, let $U$ be a bounded open set in $\C$ and let $U_\eps=\{ z\in \C:\dist(z,U) <2\eps\}$. Let $p$ be a monic polynomial such that $|p(z)|>\eps$ for all $z\in U_\eps\backslash U$. Suppose that every root of $p$ is simple and that $|s-t|\geq 2$ for any two distinct roots of $p$. Let $q$ be an analytic function on $U_\eps$ such that $|p(z)-q(z)| <\eps$ for all $z\in U_\eps$. If $\lam\in U_\eps$ and $p(\lam)=0$ then $\lam\in U$ and there exists exactly one zero of $p$ and one zero of $q$ inside the circle $C_{\lam,\eps}=\{w\in U_\eps:|w-\lam|=\eps\}$. If $\lam\in U_\eps$ and $q(\lam)=0$ then $\lam\in U$ and there exists exactly one zero of $p$ and one zero of $q$ inside $C_{\lam,\eps}$.
\end{lemma}

\begin{proof}
The assumptions of the lemma imply immediately that neither $p$ nor $q$ can vanish in $U_\eps\backslash U$. Suppose that $\lam\in U$ and $p(\lam)=0$. Then $|s-\lam|\geq 2$ for all $s$ in the finite set $S$ of roots of $p$ that are not equal to $\lam$. If $w\in C_{\lam,2\eps}$ then $|w-s|>1$ for all $s\in S$. Therefore
\[
|p(w)|=\left|(w-\lam)\prod_{s\in S}(w-s)\right|> |w-\lam|=2\eps.
\]
Since $C_{\lam,2\eps}$ and its interior are contained in $U_\eps$, we may apply Rouch\'{e}'s theorem to $p$ and $q$ on and inside $C_{\lam,2\eps}$, and deduce that $q$ has exactly one zero inside $C_{\lam,2\eps}$. The same argument evidently applies to $C_{\lam,\eps}$.

On the other hand if $\lam\in U$ and $q(\lam)=0$ then $|p(\lam)|<\eps$. If $T$ is the set of all roots of $p$ then
\[
\eps>|p(\lam)|=\left|\prod_{t\in T}(\lam-t)\right|
\]
so $|\lam-t|<1$ for at least one root of $p$; from this point onwards we use the symbol $t$ to refer to one such root. If $S=T\backslash \{t\}$ then $|\lam-s|> 1$ for all $s\in S$, so $t$ is unique. Therefore
\[
\eps>|p(\lam)|=\left|(\lam-t)\prod_{s\in S}(\lam-s)\right|> |\lam-t|.
\]
Repeating the first paragraph of the proof with $\lam$ replaced by $t$, both $p$ and $q$ have exactly one root inside the circle $C_{t,2\eps}$. This contains the region inside $C_{\lam,\eps}$, so both $p$ and $q$ have at most one root inside $C_{\lam,\eps}$. The proof is concluded by noting that we have already observed that they have at least one root inside $C_{\lam,\eps}$.
\end{proof}

\begin{example}\label{trunc3}
Given positive integers $M_1$, $M_2$ and $L$, let $N=M_1+M_2$ and let $A$ be the diagonal $N\times N$ matrix with entries
\begin{eqnarray*}
A_{r,r}&=&(r-1)/(M_1-1),\\
A_{M_1+s,M_1+s}&=&L+1+(s-1)/(M_2-1),
\end{eqnarray*}
for $1\leq r\leq M_1$ and $1\leq s\leq M_2$, so that $\Spec(A)\subset [0,1]\cup [L+1,L+2]$. Also let $A_\gam=A+\gam B$ where $B$ is the rank one operator associated with the unit vector $e\in \C^N$ defined by $e_r=(2M_1)^{-1/2}$ for $1\leq r\leq M_1$ and $e_{M_1+s}=(2M_2)^{-1/2}$ for $1\leq s\leq M_2$. One sees immediately that $\norm e\norm =1$ in $\C^N$.

The continuous curves in Figure~\ref{specgapfig} show parts of six of the spectral curves of $A_\gam$ for $\gam=t\rme^{i\theta}$ when $M_1=5$, $M_2=25$, $L=4$, $0\leq t<20$ and $\theta=89^\circ$. Most of the curves stay within a small distance of their starting point as $t$ increases. The curve starting at the eigenvalue $0.5$ of the $30\times 30$ matrix $A$ moves rapidly away from the real axis as $t$ increases but eventually converges to $3$. There is only one curve that diverges to $\infty$ as $t\to\infty$, and a part of this appears in the top right-hand part of the figure.

The dashed curves in Figure~\ref{specgapfig} are produced in a similar manner but with $M_1=5$ and $M_2=1$, so that $\tA$ is a $6\times 6$ matrix. Following the prescription of Theorem~\ref{trunc2}, we define $\tA_{r,r}$ as above for $1\leq r\leq 5$, but put $\tA_{6,6}=5.5$, so that the function $\widetilde{m}$ of Theorem~\ref{trunc2} is the Herglotz function for the pair $\tA$, $\widetilde{e}$, where $\widetilde{e}_r=(10)^{-1/2}$ for $1\leq r\leq 5$ and $\widetilde{e}_6=2^{-1/2}$. In spite of the substantial reduction in the size of the matrix, the part of the spectrum in $\{\lam:\Re(\lam)\leq 2\}$ is almost unchanged, as predicted by Theorem~\ref{trunc2}.
\end{example}

\begin{figure}[h!]
\begin{center}
\draft{\resizebox{15cm}{!}{\mbox{\rotatebox{0}
{\includegraphics{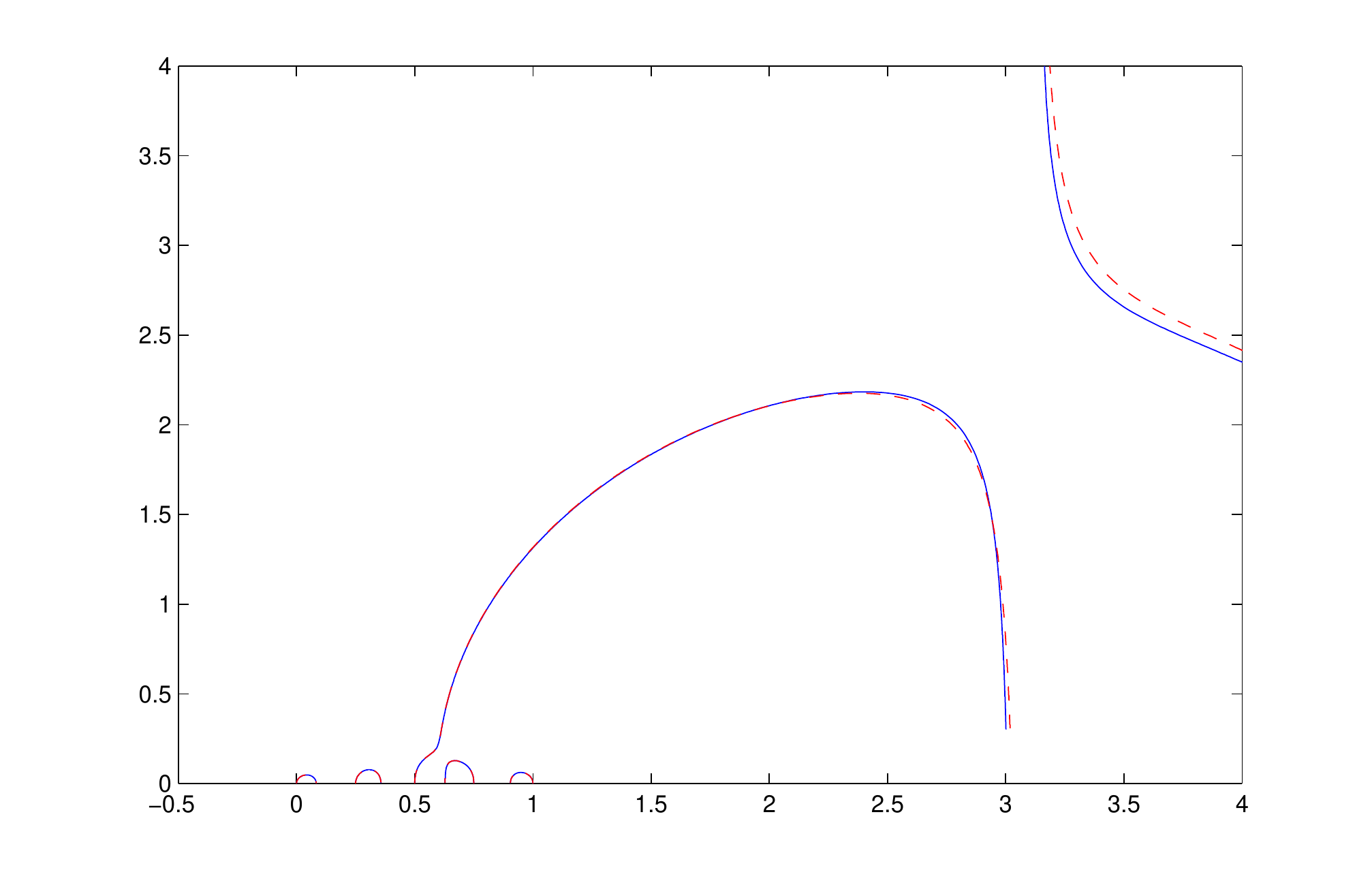}}}}}
\end{center}
\caption{Spectral curves in Example~\ref{trunc3}}\label{specgapfig}
\end{figure}

\section{Rank one perturbations}\label{rank1case}

In this section we obtain more detailed results of the type already considered under the assumptions that $A$ is a self-adjoint operator acting on the Hilbert space $\cH$ and that
$Bf =\la f ,e\ra e$ for all $f \in \cH$, where $e$ is a unit vector in $\cH$.

We define $A_\gam$ on $\cH$ by
\begin{equation}
A_\gam  f  =A  f +\gam \la  f ,e\ra e\label{Agamdef}
\end{equation}
where $\gam\in\C$. We summarize a few of the many results known in
the case $\gam\in\R$ and then consider non-real $\gam$, for which
new issues arise. We also assume that $e$ is a
cyclic vector for $A$ in the sense that $\norm
e\norm=1$ and $\cL=\lin\{A^ne:n=0,1,\ldots\}$ is dense in $\cH$. This is equivalent to $B$ being cyclic for $A$ by Corollary~\ref{cycliccor}.

The four propositions below provide the general context within which our more detailed results are proved. The first is classical and may be found in \cite{Simon}.

\begin{proposition}\label{theoremA}
Let $e$ be a cyclic vector for the bounded self-adjoint operator $A$
and let $A_\gam$ be defined by (\ref{Agamdef}). If $\gam\in\R$ then
every eigenvalue of $A_\gam$ has multiplicity one. If $\alp\in\R$
and $\lam_{\alp}$ is an isolated eigenvalue of $A_{\alp}$, then
$\lam_\alp$ can be analytically continued to all real $\gam$ that
are close enough to $\alp$ and $\lam_\gam^\pr>0$ for all such
$\gam$.
\end{proposition}

\begin{proposition}\label{lemmaB}
If $\lam\in \C\backslash \Spec(A)$ and $\gam\in\C$ then $\lam$ is
an eigenvalue of $A_\gam$ if and only if
\begin{eq}
1+\gam\la (A-\lam I)^{-1}e,e\ra=0.\label{gamlamrel}
\end{eq}
This formula defines $\gam$ as an analytic function of $\lam\in \C\backslash \Spec(A)$;
one has $\gam(\lam)=-1/m(\lam)$, where
\begin{equation}
m(\lam)=\la (A-\lam I)^{-1}e,e\ra.\label{mfunction}
\end{equation}
\end{proposition}

This is a special case of Lemma~\ref{BS}. In finite dimensions one may alternatively use the formula
\begin{eqnarray}
\frac{\det(A+\gam B-\lam I)}{\det(A-\lam I)}&=&\det((I+\gam (A-\lam I)^{-1}B)^\nat)\nonumber\\
&=& 1+\gam\la (A-\lam I)^{-1}e,e\ra . \label{rank1det}
\end{eqnarray}
See Lemma~\ref{relativedet}.

\begin{proposition}\label{theoremC}
The function $m(\lam)$ defined for all $\lam\in \C\backslash \Spec(A)$ by
(\ref{mfunction}) is a Herglotz function in the sense that $m(\lam)\in \C_\pm$ for all $\lam\in \C_\pm$. Moreover $|m(x+iy)|<1/|y|$ for all $x\in\R$
and $y\not= 0$. If $A$ is bounded then $m(\lam)\not=0$
for all $\lam\in\C$ such that $|\lam|>\norm A \norm$. It follows that
\[
\gam(\lam)=\lam+\la Ae,e\ra +O(|\lam|^{-1})
\]
as $|\lam|\to\infty$.
\end{proposition}

\begin{proposition}\label{theoremD}
Suppose that $\gam_0,\, \lam_0\in \C_+$ satisfy
(\ref{gamlamrel}) and that $\lam_0$ has algebraic multiplicity $1$
as an eigenvalue of $A_{\gam_0}$. Then there exists an  analytic
function $\lam$ of $\gam$ defined for all $\gam$ in some
neighbourhood of $\gam_0$ such that $\gam$ and $\lam$ satisfy
(\ref{gamlamrel}). Moreover $\lam^\pr(\gam)\not= 0$ in this
neighbourhood.
\end{proposition}

\begin{proof}
The first statement of the proposition is a standard fact from perturbation theory for the eigenvalues of operators that depend analytically on a parameter. Given this, we differentiate (\ref{gamlamrel}) with respect to $\gam$ to obtain
\[
\la (A-\lam I)^{-1}e,e\ra-\gam \lam^\pr(\gam) \la (A-\lam
I)^{-2}e,e\ra=0.
\]
Assuming that the neighbourhood is small enough, $\gam\notin\R$ and
$\la (A-\lam I)^{-1}e,e\ra\not= 0$ by Pro[position~\ref{theoremC}. This
implies that $\lam^\pr(\gam)\not= 0$.

\end{proof}

In the rest of this section we assume that  $N=\dim(\cH)<\infty$, that
$e\in\cH$ has norm one and is a cyclic vector for $A$, and that $\Im(\gam)\geq 0$. Our goal is to understand how the eigenvalues of $A_\gam$ depend on $\gam$ for very small and very large $\gam$, and the mapping properties from the one asymptotic regime to the other. We start with the case in which $\gam$ is real and positive.

\begin{lemma}
Under the assumptions of the last paragraph, let $\lam_1,\ldots,
\lam_N$ be the eigenvalues of $A$ written in increasing order and
let $\del_1,\ldots,\del_{N-1}$ be the eigenvalues of the truncation $A^\nat$ of $A$ to $\cK=\{  f :\la  f ,e\ra =0\}$. Then
\[
\lam_1<\del_1<\lam_2<\ldots <\del_{N-1}<\lam_N.
\]
If one assumes that $\gam\geq 0$ then  the eigenvalues
$\lam_{r,\gam}$ of $A_\gam$ are all strictly increasing analytic
functions of $\gam$ satisfying $\lam_{r,0}=\lam_r$. Moreover
$\lim_{\gam\to +\infty}\lam_{r,\gam}=\del_r$ for $1\leq r\leq N-1$
and $\lim_{\gam\to +\infty}\lam_{N,\gam}=+\infty$.
\end{lemma}

\begin{proof}
This uses Proposition~\ref{theoremA} and the variational formula for the eigenvalues of $A_\gam$.
\end{proof}

We now turn to the study of the case $\gam\in\C_+$.

\begin{theorem}\label{fibration}
Given $\theta \in (0,\pi)$, define
\begin{equation}
S_\theta=\bigcup_{t>0} \Spec(A_{t\rme^{i\theta}}).\label{SthetaA}
\end{equation}
Then
\begin{equation}
S_\theta\cap S_\phi=\emptyset\mbox{ if }\theta\not= \phi\label{SthetaB}
\end{equation}
and
\begin{equation}
\bigcup_{\theta\in (0,\pi)}S_\theta=\C_+.\label{SthetaC}
\end{equation}
Moreover the limit set of each $S_\theta$ in $\C_+\cup \{\infty\}$ is
$\Spec(A)\cup\Spec(A^\nat)\cup\{\infty\}$.
\end{theorem}

\begin{proof}
If $\lam\in\C_+$ then (\ref{gamlamrel}) determines $\gam=t\rme^{i\theta}$ uniquely. This fact implies (\ref{SthetaB}) and (\ref{SthetaC}). The limit set of $S_\theta$ is the union of $\lim_{t\to 0}\Spec(A_{t\rme^{i\theta}})$ and $\lim_{t\to +\infty}\Spec(A_{t\rme^{i\theta}})$, both of which were determined in Theorem~\ref{maintheorem}.
\end{proof}

We now turn to the structure of the individual sets $S_\theta$. Let $\del_1,\ldots,\del_{N-1}$ denote the eigenvalues of $A^\nat$, written in increasing order and let $\del_N=\infty$. As before we say that a curve $\sig:(0,\infty)\to \C_+$ is simple and analytic if it is a one-one, real analytic mapping and $\sig^\pr(t)$ is non-zero for all $t\in (0,\infty)$.

\begin{theorem}\label{main2}
There exists a finite increasing set $T\subset (0,\pi)$ such that if $\theta\in (0,\pi)\backslash T$ then $S_\theta$ is the union of $N$ disjoint simple analytic curves. Each curve starts at some $\lam_r\in\Spec(A)$ and ends at some $\del_{\tau(r)}$ where $\tau$ is a permutation of $\{1,2,\ldots,N\}$. This permutation is constant in each subinterval $J$ of  $(0,\pi)\backslash T$, but it may change from one interval to another.
\end{theorem}

\begin{proof}
This a corollary of Theorems~\ref{maintheorem} and \ref{fibration}, but some of the calculations are simpler because the polynomial $p(\gam,\lam)$ defined in (\ref{pgl}) has the following explicit form. By expanding the determinant using an orthonormal basis whose first term is $e$ one obtains
\begin{eqnarray}
p(\gam,\lam)&=&\det(A-\lam I)+\gam \det(A^\nat -\lam I^\nat)\label{pgrA}\\
&=& p_0(\lam) +\gam p_1(\lam),  \label{pgrB}
\end{eqnarray}
where ${}^\nat$ denotes the truncation to $\cK=\{\phi :\la \phi,e\ra=0\}$, $p_0$ is a polynomial with degree $N$ and $p_1$ is a polynomial with degree $N-1$. The formula (\ref{pgrB}) can also be derived from (\ref{gamlamrel}). Let $F$ be the finite exceptional set defined just before Theorem~\ref{maintheorem}. It follows from (\ref{SthetaB}) and (\ref{SthetaC}) that there is a finite set $T\subset (0,\pi)$ such that $\theta\in T$ if and only if $t\rme^{i\theta}\in F$ for some $t>0$. If $\theta\in (0,\pi)\backslash T$ then the curve $g(t)=t\rme^{i\theta}$ lies in $\cG_0$, as defined just before Theorem~\ref{maintheorem}, which yields most of the statements of this theorem. $\theta \notin T$ implies that the curves are simple and non-intersecting because every $\lam\in \C_+$ is associated with only one $\gam=t\rme^{i\theta}$ and hence with only one value of $t\in (0,\infty)$ by (\ref{gamlamrel}). The constancy of the permutation on each subinterval $J$ follow from the continuous dependence of the curves in $S_\theta$ on $\theta$. To prove the last statement, it is sufficient to consider the following example.
\end{proof}

\begin{example}\label{rank1example}
Let $\cH=\C^2$ and let
\[
A_\gam=\mtrx{ 1+\gam \alp^2& \gam \alp\bet\\
\gam\alp\bet&-1+\gam\bet^2}
\]
where $\gam\in\C$, $\alp>0$, $\bet>0$ and $\alp^2+\bet^2=1$, so that
$e=(\alp,\bet)$ has norm one and is a cyclic vector for
$A=A_0$. The eigenvalues of $A_\gam$ are
\[
\lam_{\pm,\gam}=\frac{\gam}{2}%
\pm\sqrt{  \frac{\gam^2}{4}+1-\gam(\bet^2-\alp^2)  }.
\]
Figure~\ref{rank1fig} plots these eigenvalues for $\alp=\sqrt{3}/2$, $\bet =1/2$, $\gam=t\rme^{i \theta}$ and $0\leq t\leq 4$. The two dashed curves correspond to the choice $\theta= 119^\circ$; the two intersecting continuous curves correspond to the choice $\theta= 120^\circ$; the two dotted curves correspond to the choice $\theta= 121^\circ$. Note that the only critical point is $\gam_{\rm c}=\{-1+ i\sqrt{3}\}$, so $T=\{120^\circ\}$. The corresponding eigenvalue of $A_{\gam_{\rm c}}$ is $\lam_{\rm c}=\{-1/2+ i\sqrt{3}/2\}$, which has algebraic multiplicity $2$ but geometric multiplicity $1$, by a direct computation or Theorem~\ref{SBimpliesC+}. It is clear that the permutation $\tau$ of the set $\{1,2\}$ defined in Theorem~\ref{main2} is different for $\theta< 120^\circ$ and for $\theta> 120^\circ$ and that there is no natural way of defining such a permutation for $\theta = 120^\circ$.
\end{example}

\begin{figure}[h!]
\begin{center}
\draft{\resizebox{15cm}{!}{\mbox{\rotatebox{0}%
{\includegraphics{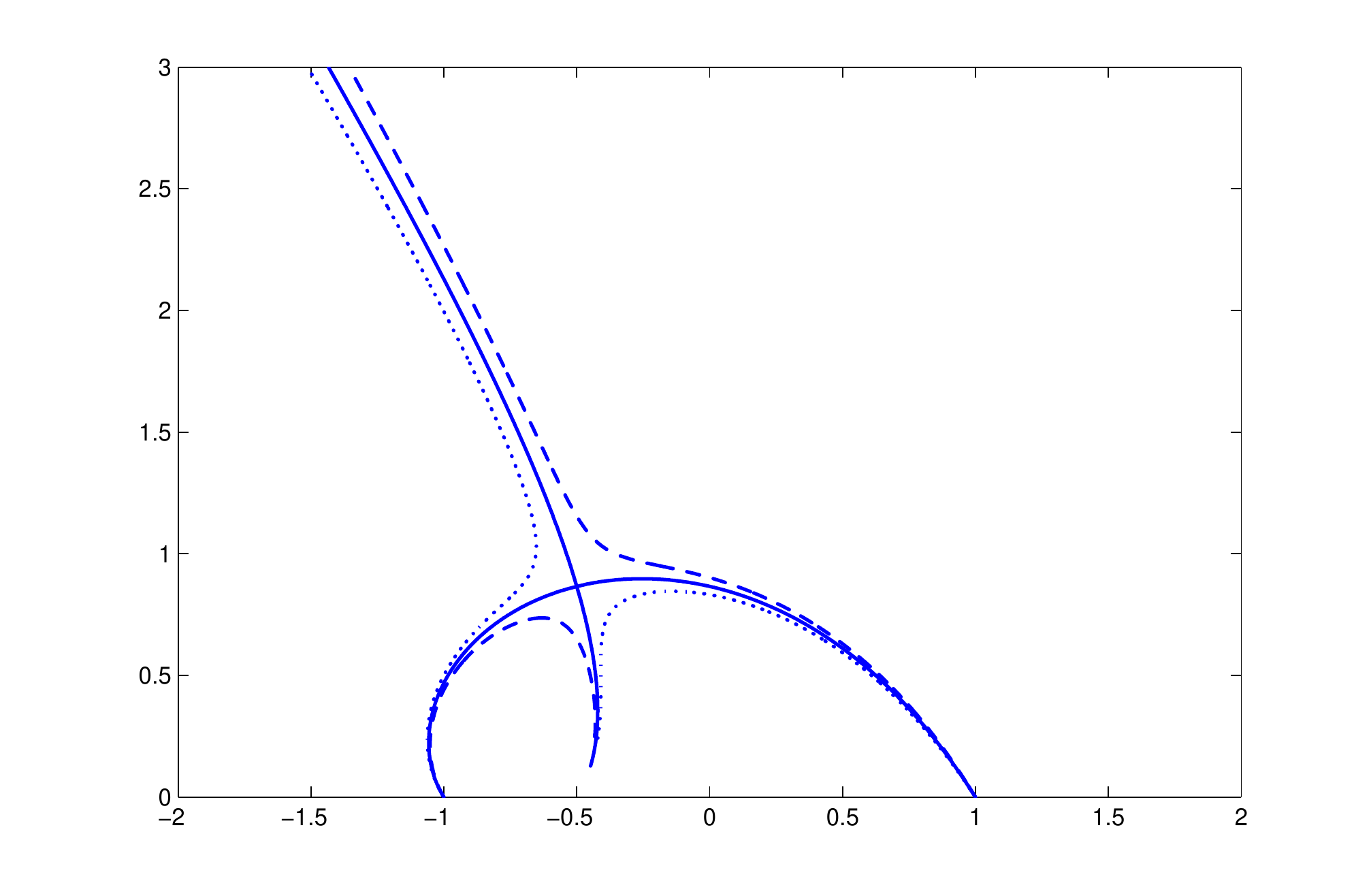}}}}}
\end{center}
\caption{Spectral curves described in %
Example~\ref{rank1example}}\label{rank1fig}
\end{figure}

\begin{example}\label{Neq5example}
Let $A$ be the $5\times 5$ diagonal matrix with eigenvalues
$\lam_r=r$, $r=1,2,3,4,5$, and let $e=(1,1,1,1,1)/\sqrt{5}$. Then
the limits of the eigenvalues of $A_\gam$ as $|\gam|\to\infty$ are
given numerically by $\mu_1= 1.35556$, $\mu_2=2.45608$,
$\mu_3=3.54390$, $\mu_4=4.64442$ and $\mu_5=\infty$. For each $\theta$ exactly one of the five eigenvalue curves diverges to $\infty$. The table below lists the permutations $\tau$ associated with each angle $\theta\in(0^\circ,180^\circ)$ that is a multiple of $10^\circ$.

\[
\begin{array}{cccccc}
\theta&\tau(1)&\tau(2)&\tau(3)&\tau(4)&\tau(5)\\
0^\circ &1&2&3&4&5\\
10^\circ &1&2&3&4&5\\
20^\circ &1&2&3&4&5\\
30^\circ &1&2&3&4&5\\
40^\circ &1&2&3&4&5\\
50^\circ &1&2&3&4&5\\
60^\circ &1&2&3&4&5\\
70^\circ &1&2&3&5&4\\
80^\circ &1&2&3&5&4\\
90^\circ &1&2&5&3&4\\
\end{array}
\]
The first change in $\tau$ occurs for $\theta_1\in
(61^\circ,62^\circ)$, while the second occurs for $\theta_2\in
(81^\circ,82^\circ)$.
\end{example}

\section{The limit $N\to\infty$}\label{dimdep}
The previous analysis clarifies to some extent how the spectra of a family of $N\times N$ matrices $A_\gam=A+\gam B$ depend on $\gam$ for very small and very large $\gam$. However, it does not capture the full range of phenomena that can occur for $\gam$ of intermediate sizes. Even if one is interested in a particular fairly large value of $N$, one often obtains further insights by considering a family of $N\times N$ matrices $A_{N,\gam}=A_N+\gam B_N$. From this point of view the case $N =\infty$ is regarded as an idealization that may be simpler to analyze than the original problem. Results such as Proposition~\ref{intapprox} may then be used to estimate the difference between the two cases.

We assume throughout that $B_Nf=\la f,e_N\ra e_N$ for all $f\in\C^N$ where $e_N\in\C^N$ is a unit vector. The set of eigenvalues of $A_{N,\gam}$ is obtained by solving $1+\gam m_N(\lam)=0$, or equivalently $\gam=-1/m_N(\lam)$, where $m_N$ are the Herglotz functions
\[
m_N(\lam)=\la (A_N-\lam I_N)^{-1}e_N,e_N\ra.
\]
See Propositions~\ref{lemmaB} and \ref{theoremC}.

\begin{theorem}\label{Ntoinftytheorem}
Let
\[
\mu_N(\gam)=\max\left\{ \Im(\lam_{r,N,\gam}):1\leq r\leq N\right\}
\]
where $\gam\in \C_+$ and $\{\lam_{r,N,\gam}\}_{r=1}^N$ are the eigenvalues of $A_{N,\gam}$ repeated according to their algebraic multiplicities. Then
\begin{equation}
\mu_N(\gam)\geq \Im(\gam)/N\label{NtoinfinityLB}
\end{equation}
for all $N\geq 1$ and $\gam\in\C_+$. Suppose that $\norm A_N\norm\leq c$ for all $N$ and that $m_N$ converge to $m_\infty$ locally uniformly on $\C_+$ as $N\to\infty$. If $\gam\in\C_+$ and $-1/\gam\notin\Ran(m_\infty)$ then
\[
\lim_{N\to\infty} \mu_N(\gam)=0.
\]
\end{theorem}

\begin{proof}
We note that $m_\infty$ is a Herglotz function, unless it is a real constant. Each function $m_N:\C_+\to \C_+$ is surjective because $m_N(\lam)=-1/\gam$ has $N$ solutions $\lam\in\C_+$ counting multiplicities, namely the eigenvalues of $A_{N,\gam}$. We show in Example~\ref{N2inftyex} that $m_\infty$ need not be surjective. The lower bound (\ref{NtoinfinityLB})
follows from
\[
N\mu_N(\gam)\geq \Im\left(\sum_{r=1}^N\lam_{r,N,\gam}\right)=\Im\left(\tr(A_N +\gam B_N )\right) =\Im(\gam).
\]
If $\gam\in\C_+$ then every eigenvalue $\lam_{r,N,\gam}$ of $A_{N,\gam}$ satisfies
\begin{equation}
|\lam_{r,N,\gam}|\leq \norm A_N\norm +|\gam|\, \norm B_N\norm \leq c+|\gam|.\label{ABbound}\
\end{equation}
Therefore $0<\mu_N(\gam)\leq c+|\gam|$. If $\mu_N(\gam)$ does not converge to $0$ as $N\to\infty$ then there exists a subsequence $N(s)$ and a constant $c_2>0$ such that
$\mu_{N(s)}(\gam)\geq c_2$ for all $s$; and then a subsequence $r(s)$ such that $\Im(\lam_{r(s),N(s),\gam})\geq c_2$ for all $s$. Combining this with (\ref{ABbound}), there exist subsubsequences, which we again parametrize using $s$, and $\lam\in\C_+$ such that
$\lim_{s\to\infty}\lam_{r(s),N(s),\gam}=\lam$ where $\Im(\lam)\geq c_2$ and $|\lam |\leq c+|\gam|$. Since $m_{N(s)}(\lam_{r(s),N(s),\gam})=-1/\gam$ for all $s$, the local uniform convergence of $m_N$ to $m_\infty$ implies that $m_\infty(\lam)=-1/\gam$.
\end{proof}

Theorem~\ref{Ntoinftytheorem} depends on the assumption that $m_N$ converges to $m_\infty$ as $N\to\infty$. The following proposition allows one to estimate the difference between $m_N(\lam)$ and $m_\infty(\lam)$ for problems of the above type by putting
\[
k(s)=\frac{|g(s)|^2}{s-\lam},
\]
where the choice of $g$ depends on the problem. It may be seen that the bound on the difference is $O(\Im(\lam)^{-2})$ as $\Im(\lam)\to 0$.

\begin{proposition}\label{intapprox}
Let $k$ be a continuous function on $[a,b]$ with bounded first derivative and let $N$ be a positive integer. Then
\begin{equation}
\frac{b-a}{N}\sum_{r=1}^N k(a+r(b-a)/N)=\int_a^b k(s)\, \rmd s+\mathrm{rem}\label{sumapprox}
\end{equation}
where
\[
|\mathrm{rem}|\leq \frac{(b-a)^2}{2N}\norm k^\pr \norm_\infty.
\]
\end{proposition}

\begin{proof}
The left hand side of (\ref{sumapprox}) is the sum of $N$ terms of the form
\begin{eqnarray*}
\del k(c_r)&=& \int_{c_r-\del}^{c_r}\frac{\rmd}{\rmd s}[(s-c_r+\del)k(s)]\, \rmd s\\
&=& \int_{c_r-\del}^{c_r}k(s)\, \rmd s+C_r
\end{eqnarray*}
where $\del=(b-a)/N$, $c_r=a+r\del$ and
\begin{eqnarray*}
|C_r|&=&\left| \int_{c_r-\del}^{c_r}(s-c_r+\del)k^\pr(s))\, \rmd s\right|\\
&\leq & \norm k^\pr\norm_\infty \int_{c_r-\del}^{c_r}|s-c_r+\del |\, \rmd s\\
&=& \frac{\del^2}{2} \norm k^\pr\norm_\infty.
\end{eqnarray*}
Summing over $r$ one obtains
\[
|\mathrm{rem}|\leq \frac{N\del^2}{2} \norm k^\pr\norm_\infty %
=\frac{(b-a)^2}{2N} \norm k^\pr\norm_\infty.
\]
\end{proof}

\begin{example}\label{N2inftyex}
Let $A_N$ be the $N\times N$ diagonal matrix with entries $A_{N,n,n}=n/N$ for all $n$ and let $B_N$ the rank one matrix associated with the unit vector $e_{N,n}=N^{-1/2}$ for all $n$. Then $\Spec(A_N)\subset [0,1]$ and
\[
m_N(\lam)=\frac{1}{N} \sum_{r=1}^N \frac{1}{n/N-\lam}.
\]
It may be seen that $m_N$ converges locally uniformly to
\begin{eq}
m_\infty(\lam)=\int_0^1 \frac{\rmd s}{s-\lam}
= \log\left(\frac{\lam -1}{\lam}\right)
\end{eq}
as $N\to\infty$. It follows that the range of $m_\infty$ is $\{z\in\C :0<\Im(z)<\pi\}$ and the set of $\mu\in\C_+$ such that $m_\infty(\lam)=-1/\mu$ has no solution is the closed disc
\[
D= \{ \gam :\left|\gam-i/(2\pi)\right|\leq 1/(2\pi)\} .
\]

Figure~\ref{Ntoinftyfig} provides a contour plot $\mu_N(\gam)$ for $N=100$, the contours corresponding to the values $0.01,\, 0.02,\, 0.03,\, 0.04,\, 0.05$ of $\mu_N(\gam)$. The circle $\partial D$ is included for comparison.
\end{example}

\begin{figure}[htb!]
\begin{center}
\draft{\resizebox{15cm}{!}{\mbox{\rotatebox{0}%
{\includegraphics{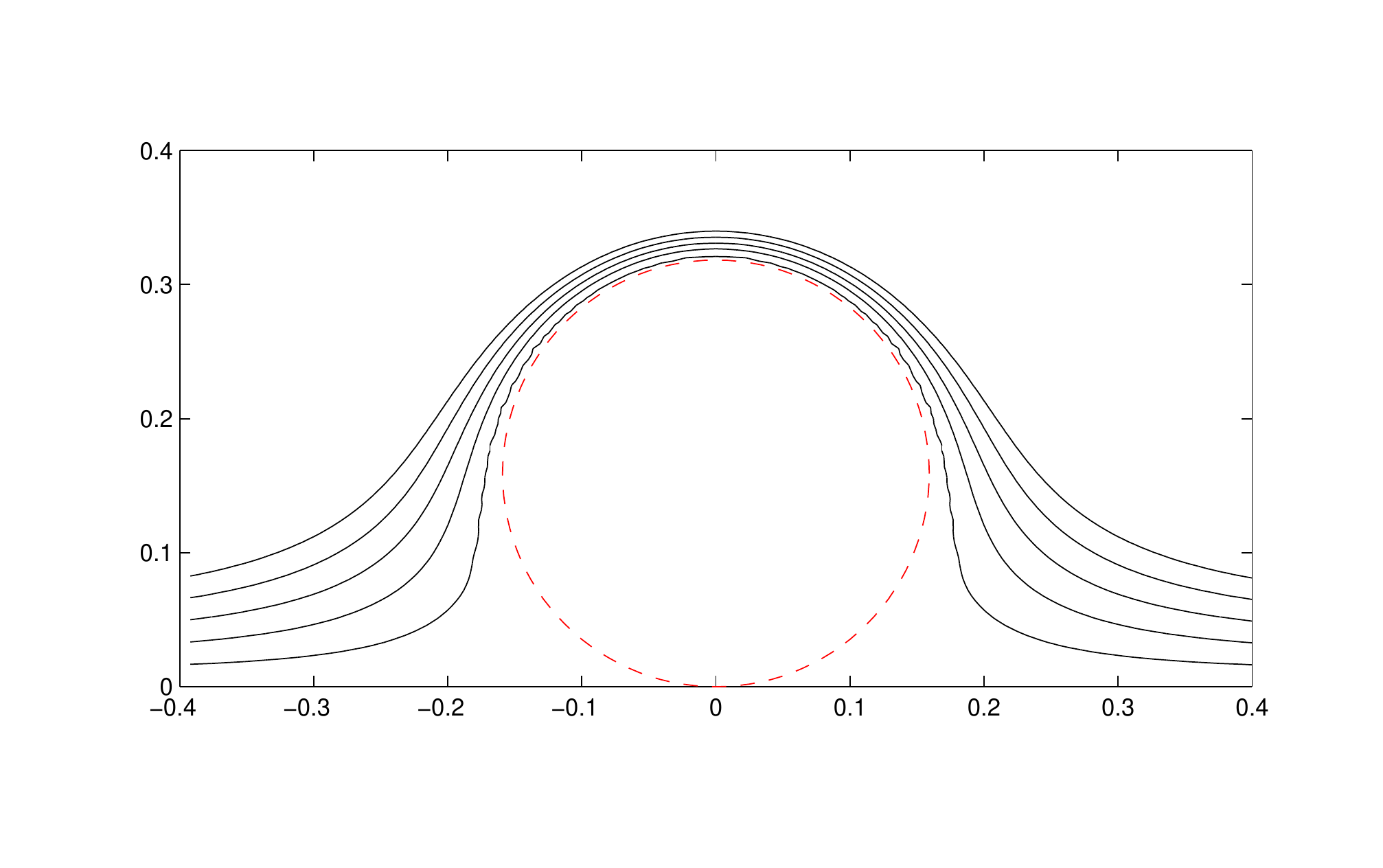}}}}}
\end{center}
\caption{Contour plot of $\mu_N(\gam)$ for $N=100$.}\label{Ntoinftyfig}
\end{figure}

\begin{example}\label{generalg}
Let $A_N$ be the $N\times N$ diagonal matrix with entries $A_{N,n,n}=n/N$ for all $n$ and let $B_N$ the rank one matrix associated with the unit vector $e_{N,n}=N^{-1/2}g(n/N)$ where $g\in L^2(0,1)$ and $g$ is sufficiently regular. Then $\Spec(A_N)\subset [0,1]$ and $m_N$ converges locally uniformly to
\begin{eq}
m_\infty(\lam)=\int_0^1 \frac{f(s)\rmd s}{s-\lam},\label{generalgeq}
\end{eq}
as $N\to\infty$, where $f(s)=|g(s)|^2$. The integral in (\ref{generalgeq}) is well-defined for every $\lam\in\C_+$ because $f\in L^1(0,1)$, but the form of the range of $m_\infty$ depends on whether either or both of the integrals
\[
\int_0^1 \frac{f(s)}{s}\,\rmd s, \hspace{2em} \int_0^1 \frac{f(s)}{1-s}\,\rmd s
\]
is finite. In Example~\ref{N2inftyex} both integrals are infinite. More generally the first integral diverges if and only if $A+\gam B$ has a negative eigenvalue for all real negative $\gam$, while the second integral diverges if and only if $A+\gam B$ has a positive eigenvalue for all real positive $\gam$

The range of $m_\infty$ is the union of the sets $m_\infty(S_{\eps,r})$, where
\[
S_{\eps,r}= \left\{ \lam: \Im(\lam)\geq\eps \mbox{ and } |\lam|\leq r    \right\}.
\]
These increase monotonically as $\eps>0$ decreases to $0$ and as $r$ increases to $\infty$. The boundary $\partial S_{\eps,r}$ may be parametrized as a simple closed curve $\gam_{\eps,r}$. A standard theorem in complex analysis states that each set
$m_\infty(S_{\eps,r})$ is the union of the range of the closed curve $\sig_{\eps,r}=m_\infty\circ \gam_{\eps,r}$ and the set of all $z$ not in this range whose winding number with respect to $\sig_{\eps,r}$ is non-zero.

The observations above allow one to compute the range of $m_\infty$ approximately by taking $\eps>0$ small enough and $r$ large enough. This is particularly easy if one can write $m_\infty$ in closed form. If $g(s)=s^{1/2}$ for all $s\in [0,1]$ then
\[
m_\infty(\lam)=\int_0^1 \frac{s}{s-\lam}\, \rmd s=1+\lam\log\left( \frac{\lam-1}{\lam}\right).
\]
Figure~\ref{generalgfig} was obtained by putting $\eps=10^{-8}$. The set of $\gam$ for which $-1/\gam=m_\infty(\lam)$ is not soluble is the part of $\C_+$ that is inside the closed curve plotted. This curve starts at $-1$ and ends at $0$. The gap observed near $0$ is a numerical artifact that arises because the convergence to $0$ is logarithmic.
\end{example}

\begin{figure}[h!]
\begin{center}
\draft{\resizebox{15cm}{!}{\mbox{\rotatebox{0}
{\includegraphics{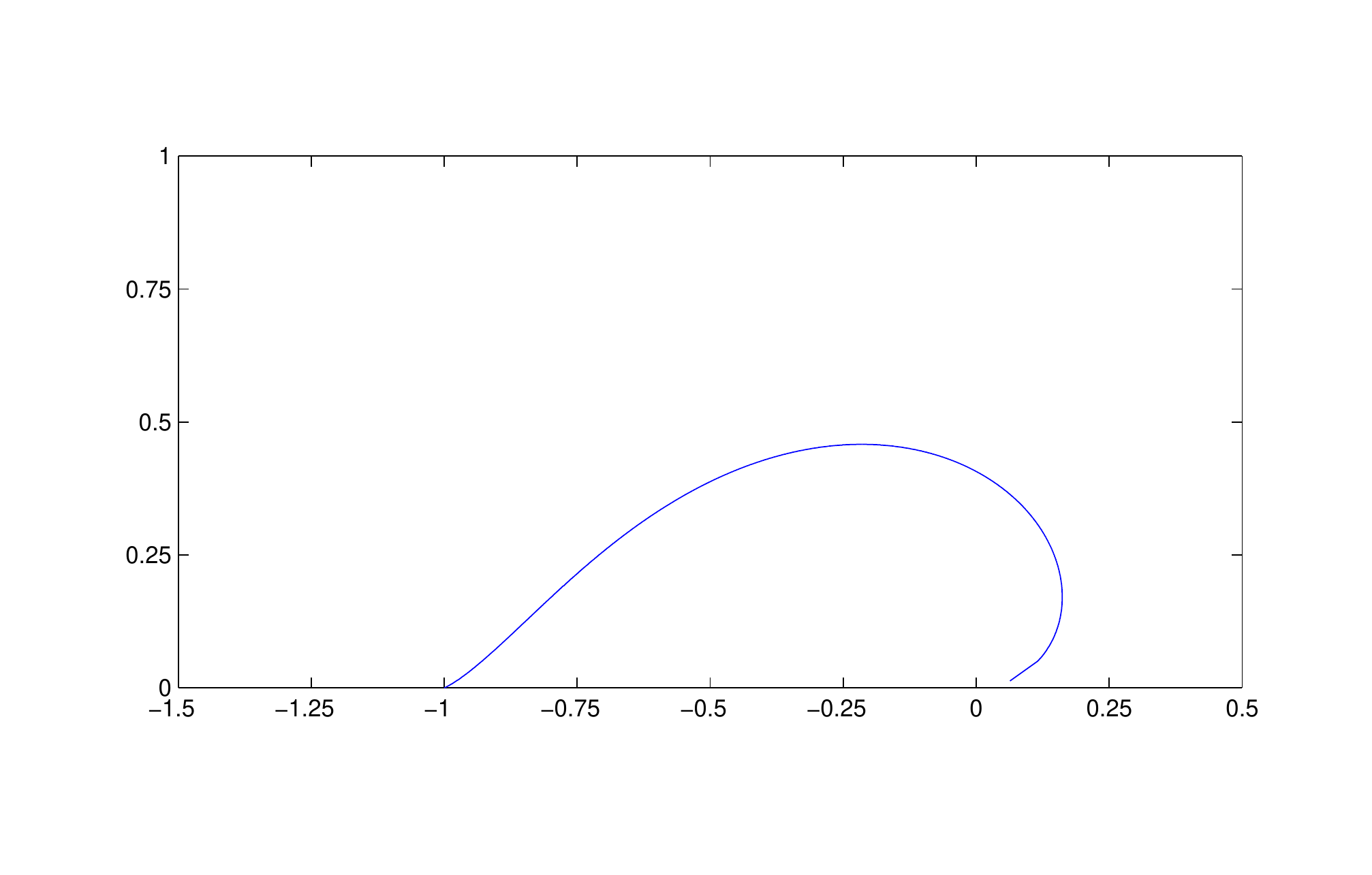}}}}}
\end{center}
\caption{Boundary curve in Example~\ref{generalg}}\label{generalgfig}
\end{figure}

\begin{theorem}\label{mrange}
Suppose that
\begin{equation}
m(\lam)=\int_\R \frac{f(s)}{s-\lam}\, \rmd s\label{mintdef}
\end{equation}
for all $\lam\in \C_+$, where $f(s)\geq 0$ for all $s\in\R$,
\begin{equation}
\int_\R \frac{f(s)}{1+|s|}\, \rmd s< \infty\label{mfinitecond}
\end{equation}
and $0<\norm f\norm_\infty <\infty$.
Then $m(\C_+)$ is contained in
\begin{equation}
\{ z:0<\Im(z)<\pi\norm f\norm_\infty\}.\label{mrangeA}
\end{equation}
Hence $-1/\gam= m(\lam)$ is not soluble for any $\gam$ in the closed disc
\begin{equation}
\left\{ \gam :\left|\gam-\frac{i}{2\pi\norm f\norm_\infty}\right|\right\}\leq \frac{1}{2\pi\norm f\norm_\infty}\, .\label{mrangeB}
\end{equation}
\end{theorem}

\begin{proof}
The condition (\ref{mfinitecond}) ensures that the integral (\ref{mintdef}) defining $m(\lam)$ converges for all $\lam\in\C_+$ and defines an analytic function of $\lam$. The condition $f(s)\geq 0$ and $0<\norm f\norm_\infty$ ensures that the range of $m$ is contained in $\C_+$.

We next observe that if $\lam=u+iv$ where $u\in\R$ and $v>0$ then
\[
\Im(m(\lam))=\int_\R \frac{v}{(u-s)^2+v^2} f(s)\, \rmd s.
\]
A direct estimate of this integral yields (\ref{mrangeA}), and (\ref{mrangeB}) follows.
\end{proof}

\clearpage

\textbf{Acknowledgements} The author thanks F Gezstesy, A Pushnitski and Y Safarov for valuable comments and contributions.

Dept. of Mathematics,\\
King's College London,\\
Strand,\\
London,WC2R 2LS,\\
UK

\end{document}